\documentclass[11pt]{article}

\pdfoutput=1

\usepackage{amsfonts}
\usepackage{amssymb}
\usepackage{amsmath}
\usepackage{color}
\usepackage{caption}
\usepackage{amsthm}
\usepackage{graphicx}
\usepackage{framed}
\usepackage[normalem]{ulem}

\def\R{\mathbb{R}}
\def\N{\mathbb{N}}

 \numberwithin{equation}{section}

 \newtheorem{theorem}{Theorem}[section]
  
  \newtheorem{corollary}[theorem]{Corollary}
  \newtheorem{lemma}[theorem]{Lemma}
\newtheorem{remark}[theorem]{Remark}

\newcommand{\dx}{\, {\rm d} x}
\newcommand{\dtheta}{\, {\rm d} \theta}

\newcommand{\bTheta}{{\boldsymbol{\Theta}}}

\newcommand{\eps}{\varepsilon}
\newcommand{\PP}{\mathbb{P}}

\newcommand{\EE}{\mathbb{E}}
\newcommand{\VV}{\mathbb{V}}
\newcommand{\RR}{\mathbb{R}}

\newcommand{\fracs}[2]{{ \textstyle \frac{#1}{#2} }}

\definecolor{MyRed}{RGB}{204,0,0}
\definecolor{MyBlue}{RGB}{18,74,145}

\graphicspath{{./figs/}}

\begin{document}

\title{A Hierarchical Multilevel Markov Chain Monte Carlo Algorithm with  Applications to Uncertainty Quantification in Subsurface Flow\footnote{Part of this work was performed under the auspices of the U.S. Department of Energy by Lawrence Livermore National Laboratory under Contract DE-AC52-07A27344. LLNL-JRNL-630212-DRAFT}}

\author{T.J. Dodwell$^1$, \ C. Ketelsen$^2$, \ R. Scheichl$^3$ \ and \ A.L. Teckentrup$^4$} 

\date{}
\maketitle

\begin{center}
\begin{footnotesize}

\vspace{-0.75cm}

\noindent
${}^1$ Dept of Mechanical Engineering, University of Bath, Bath BA2
7AY, UK 

\vspace{0.1cm}

\noindent
${}^2$ Dept of Applied Mathematics, 526 UCB, University of Colorado at
Boulder, CO 80309-0526, USA

\vspace{0.1cm}

\noindent
${}^3$ Dept of Mathematical Sciences, University of Bath, Bath BA2
7AY, UK. \ Email: 
{\tt R.Scheichl@bath.ac.uk}

\vspace{0.1cm}

\noindent
${}^4$ Mathematics Institute, Zeeman Building, University of Warwick,
Coventry CV4 7AL, UK 

$  $
\end{footnotesize}
\end{center}


\begin{abstract}
In this paper we address the problem of the prohibitively
large computational cost of existing Markov chain Monte 
Carlo methods for large--scale applications with high 
dimensional parameter spaces, e.g. in uncertainty quantification 
in porous media flow. We propose a new multilevel Metropolis-Hastings 
algorithm, and give an abstract, problem dependent theorem on 
the cost of the new multilevel estimator based on a set of simple, 
verifiable assumptions. For a typical model problem in subsurface 
flow, we then provide a detailed analysis of these assumptions 
and show significant gains over the standard Metropolis-Hastings 
estimator. Numerical experiments confirm the analysis and 
demonstrate the effectiveness of the method with consistent 
reductions of more than an order of magnitude
in the cost of the multilevel estimator over the standard 
Metropolis-Hastings algorithm for tolerances $\varepsilon < 10^{-2}$.
\end{abstract}

\vspace{2ex}


\section{Introduction}\label{section: introduction}

The parameters in mathematical models for many physical
processes are often impossible to determine fully or accurately, and
are hence subject to uncertainty. It is of great importance to
quantify the uncertainty in the model outputs based on the 
(uncertain) information that is available on the model inputs. 
A popular way to achieve this is stochastic modelling. Based on the
available information, a probability distribution (the {\it prior} in
the Bayesian framework) is assigned to the input parameters. If in
addition, some dynamic data (or {\em observations}) $F_\mathrm{obs}$
related to the model outputs are available, it is possible to 
reduce the overall uncertainty and to get a better representation 
of the model by conditioning the prior distribution on this data 
(leading to the {\it posterior}). 

In most situations, however, the posterior distribution is intractable
in the sense that exact sampling from it is impossible. One way to 
circumvent this problem, is to generate samples using a
Metropolis--Hastings--type Markov chain Monte 
Carlo (MCMC) approach \cite{hastings70,mrrtt53,robert_casella}, 
which consists of two main steps: (i) given the previous sample, 
a new sample is generated 
according to some proposal
distribution, such as a random walk; (ii) the likelihood of this
new sample (i.e. the model fit to $F_\mathrm{obs}$) is compared to the 
likelihood of the
previous sample. Based on this comparison, the proposed sample 
is either accepted and used for inference, or rejected and 
the previous sample is used again, leading
to a Markov chain.
A major problem with MCMC is the high cost of the likelihood calculation 
for large--scale applications, e.g. in subsurface flow where, for 
accuracy reasons, a partial differential equation (PDE) with highly 
varying coefficients needs to be solved numerically on a fine spatial grid.
Due to the
slow convergence of Monte Carlo averaging, the number of samples is
also large and moreover, the likelihood has to be calculated also for 
all the samples that are rejected in the end. Altogether, this often
leads to an
intractably high overall complexity, particularly in the
context of high-dimensional parameter spaces ({typical in 
subsurface flow}), where the acceptance 
rate of MCMC methods can be very low.

We show here how the computational cost of the standard
Metropolis-Hastings algorithm can be reduced significantly by using a
multilevel approach. This has already 
proved highly successful in the context 
of standard Monte Carlo estimators based on independent and identically 
distributed (i.i.d.) samples
\cite{cgst11,bsz11,gkss11,cst11,tsgu13} for subsurface flow problems. 
The multilevel Monte Carlo 
(MLMC) method was first introduced by Heinrich for the computation 
of high-dimensional, parameter-dependent integrals \cite{heinrich01}, 
and then rediscovered by Giles \cite{giles08} in the context of 
stochastic differential equations in finance. Similar ideas were also 
used in \cite{bgr94,bi03}
to accelerate statistical mechanics calculations. 
The basic ideas are to (i) exploit the
linearity of expectation, (ii) introduce
a hierarchy of computational models that converge  
(with increasing model resolution)
to some limit model (e.g. the original PDE), and (iii) build 
estimators for {the differences of output quantities instead of 
the quantities themselves.} In the context of 
PDEs with random coefficients, the multilevel estimators use 
a hierarchy of spatial grids and exploit that the 
numerical solution of a PDE, and thus 
the evaluation of the likelihood, is computationally much 
cheaper on coarser spatial grids.
In that way, the individual estimators will either have small
variance, since differences of output quantities from
consecutive models go to zero with increased model
resolution, or they will require significantly less computational work per
sample for low model resolutions. Either way the cost of all the
individual estimators is significantly reduced, easily compensating for
the cost of having to compute $L+1$ estimators instead of one,
where $L$ is the number of levels.

However, the application of the multilevel approach in the context of 
MCMC is not straightforward. The posterior distribution, which
depends on the likelihood, has to be level-dependent, since otherwise
the cost on all levels would be dominated by the evaluation of the
likelihood on the finest level, leading to no real cost reduction. 
In order to avoid introducing extra 
bias in the estimator, we construct instead two parallel Markov chains
$\{\theta_\ell^n\}_{n\ge 0}$ and $\{\Theta_{\ell-1}^n\}_{n\ge 0}$ on
levels $\ell$ and $\ell-1$ each from the correct posterior
distribution on the respective level. The coarser of the two chains is
constructed using the standard Metropolis--Hastings algorithm, for
example using a (preconditioned) random walk. The main innovation is 
a new proposal distribution for the finer
of the two chains $\{\theta_\ell^n\}_{n\ge 0}$. Although similar 
two-level sampling strategies have been investigated in other
applications \cite{cf05,ehl06,fbwlh03}, the computationally cheaper 
coarse models were only used to accelerate the MCMC sampling and 
not as a variance reduction technique in the estimator. Some 
ideas on how to obtain a multilevel version of the MCMC estimator 
can also be found in the recent work \cite{hss12} on sparse 
MCMC finite element methods.

The central result of the paper is a complexity theorem (cf. Theorem 
\ref{thm:main}) that quantifies, for an 
abstract large--scale inference problem, the gains in the $\varepsilon$-cost 
of the multilevel Metropolis--Hastings algorithm over the standard 
version, i.e. the cost to achieve a root mean square error less than $\varepsilon$, in terms of powers of the tolerance $\varepsilon$. For a particular application 
in stationary, single phase subsurface flow with log-normal 
permeability prior and exponential covariance, we then verify 
the assumptions of Theorem \ref{thm:main}. We show that the 
$\varepsilon$-cost of our new multilevel version is indeed one order 
of $\varepsilon$ lower than its single-level counterpart (cf. 
Theorem \ref{thm:rates}), i.e. $\mathcal{O}(\varepsilon^{-(d+1)-\delta})$ 
instead of $\mathcal{O}(\varepsilon^{-(d+2)-\delta})$, for any $\delta > 0$,
where $d$ is the spatial dimension of the problem. 
The numerical experiments for $d=2$ 
in Section \ref{sec:num} confirm the theoretical results. 
In fact, in practice the cost for the multilevel estimator grows 
only like $\mathcal{O}(\varepsilon^{-d})$, but this seems to be a 
pre--asymptotic effect. The absolute cost 
is about $\mathcal{O}(\text{10--50})$ times lower than for the standard 
estimator for values of $\varepsilon$ around $10^{-3}$, 
which is a vast improvement
and brings the cost of the multilevel MCMC estimator down to a similar
order of the cost of standard multilevel MC estimators based on
i.i.d. samples. 
This provides real hope for practical applications of MCMC
analyses in subsurface flow and other large scale PDE applications. 

The outline of the rest of the paper is as follows. 
In Section \ref{sec:mcmc}, we recall, in a very general context, 
the Metropolis Hastings algorithm, together with results on its 
convergence. In Section \ref{sec:mlmcmc}, we then present a new 
multilevel version and give a general convergence analysis under 
a set of problem-dependent, but verifiable assumptions. A 
typical model problem arising in subsurface flow modelling is 
then presented in Section 4. We briefly describe the application 
of the new multilevel algorithm to this application, and give a 
rigorous convergence analysis and cost estimate of the new 
multilevel estimator {by verifying the abstract assumptions from 
Section \ref{sec:mlmcmc}}. Finally, in Section \ref{sec:num}, 
we present some numerical results for the model problem 
discussed in Section \ref{sec:mod}.


\section{Standard Markov chain Monte Carlo}
\label{sec:mcmc}
We will start with a review of the standard Metropolis Hastings
  algorithm, described in a general context. For ease of presentation,
  we leave a precise mathematical description of our model problem until Section \ref{sec:mod}. We denote by $\theta := 
(\theta_i)_{i=1}^R$ the $\RR^R$--valued random input vector to the
model, and denote by $X := (X_j)_{j=1}^M = X(\theta)$ the
$\RR^M$--valued random output. Let further $Q_{M,R} = \mathcal G(X)$
be some linear or non--linear functional of $X$. In the context of groundwater flow modelling, this could for example be the value of the pressure or the Darcy flux at or around a given point in the computational domain, or the outflow over parts of the boundary. In practice, both
$\theta$ and $X$ are often finite dimensional approximations of
infinite dimensional objects, and an underlying "true" model is
recovered as $M, R \rightarrow \infty$. We shall therefore refer to
$M$ as the {\em discretisation level} of the model. For more details see Section \ref{sec:mod}.

We consider the setting where we have some real-world
data (or {\em observations}) $F_\mathrm{obs}$ available, and want to incorporate this information into our simulation in order to reduce the overall uncertainty. The data $F_\mathrm{obs}$ is assumed to be finite dimensional, with $F_\mathrm{obs} \in \R^m$ for some $m \in \mathbb N$, and usually corresponds to another linear or non-linear functional $\mathcal F(X)$ of the model output. 

Let us denote the density of the conditional distribution of $\theta$ given $F_\mathrm{obs}$ by $\mathcal P(\theta\, | \,F_\mathrm{obs})$. Using Bayes' Theorem, we have
\begin{equation*}
\mathcal P(\theta\, | \,F_\mathrm{obs}) = \frac{\mathcal L(F_\mathrm{obs}\, |\, \theta) \, \pi_0^R(\theta)}{\mathcal P( F_\mathrm{obs})} \eqsim \mathcal L(F_\mathrm{obs}\, |\, \theta) \, \pi_0^R(\theta).
\end{equation*}
In the Bayesian framework, one usually refers to the conditional distribution $\mathcal P(\theta\, | \,F_\mathrm{obs})$ as the {\em posterior distribution}, to $\mathcal L(F_\mathrm{obs}\, |\, \theta)$ as the {\em likelihood} and to $ \pi_0^R(\theta)$ as the {\em prior distribution}. Since the normalising constant {$\mathcal P(F_\mathrm{obs})$} is not known in general, the conditional distribution $\mathcal P(\theta\, | \,F_\mathrm{obs})$ is generally intractable and exact sampling not available.

The likelihood $\mathcal L(F_\mathrm{obs}\, |\, \theta)$  gives the probability of observing the data $F_\mathrm{obs}$ given a particular value of $\theta$. In practice, this usually involves computing {the {\em model response}} $F_{M,R} := \mathcal F \left(X(\theta)\right)$ and comparing this to the observed data $F_\mathrm{obs}$. Note that since the model response depends on the discretisation parameter $M$, in practice we compute an approximation $\mathcal L_M(F_\mathrm{obs}\, |\, \theta)$ of the true likelihood $\mathcal L(F_\mathrm{obs}\, |\, \theta)$. We will denote the corresponding density of the approximate posterior distribution by
\[
\pi^{M,R}(\theta) \eqsim \mathcal L_M(F_\mathrm{obs}\, |\, \theta) \, \pi_0^R(\theta).
\]

Let now $\nu^{M,R}(\theta) := \pi^{M,R}(\theta) \dtheta$ denote the probability measure corresponding to the density $\pi^{M,R}$. We assume that as $M,R \rightarrow \infty$, we have $\EE_{\nu^{M,R}}\left[ Q_{M,R}\right] \rightarrow \EE_{\rho}\left[ Q \right]$, for some (inaccessible) random variable $Q$ and measure $\rho$. 
The goal of the simulation is to estimate $\EE_{\nu^{M,R}}\left[
    Q_{M,R} \right]$, for $M$, $R$ sufficiently large. Hence, we
  compute approximations (or {\em estimators}) $\widehat Q_{M,R}$ of
  $\EE_{\nu^{M,R}}\left[ Q_{M,R}\right]$. To estimate this with a
Monte Carlo type estimator, or in other words by a finite sample
average, we need to generate samples from the conditional distribution
$\nu^{M,R}$, which is usually intractable, as already mentioned. 
Instead, we will use the Metropolis Hastings MCMC algorithm in Algorithm 1.

\begin{figure}[t]
\begin{framed}
\noindent
{\bf ALGORITHM 1. (Metropolis Hastings MCMC)} \vspace{1ex}\\ 
{Choose $\theta^0$. For $n \geq 0$:}
\begin{itemize}
\item Given $\theta^n$, generate a proposal $\theta'$ from a given proposal distribution $q(\theta' | \theta^n)$.
\item Accept $\theta'$ as a sample with probability
\begin{equation}\label{eq:alpha}
\alpha^{M,R} \left( \theta' | \theta^n \right)  = \min \left\{ 1, \frac{\pi^{M,R}(\theta') \, q(\theta^n | \theta')}{\pi^{M,R}(\theta^n)\,  \, q(\theta'  |  \theta^n)}\right\} 
\end{equation}
i.e. $\theta^{n+1}=\theta'$ with probability $\alpha^{M,R}$ and $\theta^{n+1}=\theta^n$ with probability $1 - \alpha^{M,R}$.
\end{itemize}
\end{framed}
\end{figure}

Algorithm 1 creates a Markov chain $\{\theta^n\}_{n \in \mathbb N}$,
and the states $\theta^n$ are used as samples for inference in a Monte
Carlo sampler in the usual way. The proposal distribution $q(\theta' |
\theta^n)$ is what defines the algorithm. A common choice is a simple
random walk. However, as outlined in \cite{hsv11}, the basic random
walk does not lead to a convergence that is independent of the
  input dimension $R$. A better choice is a preconditioned
Crank-Nicholson (pCN) algorithm \cite{cds12}, which is also a crucial
ingredient in the multilevel Metropolis-Hastings algorithm applied to
the subsurface flow model problem below.

Under reasonable 
assumptions, one can show that $\theta^n \sim \nu^{M,R}$, as $n \rightarrow \infty$, and that sample averages computed with these samples converge to expected values with respect to the desired target distribution $\nu^{M,R}$ (see Theorem \ref{thm:mcconv}). 
The first few samples of the chain, say $\theta^0, \dots,
\theta^{n_0}$, are not usually used for inference to allow the chain
to get close to the target distribution $\nu^{M,R}$. This is referred
to as the {\em burn--in} of the MCMC algorithm. Although the length of
the burn-in is crucial for practical purposes, and largely influences
the behaviour of the resulting MCMC estimator for finite sample sizes,
asymptotic statements about the estimator are usually independent of the burn-in. We will denote our MCMC estimator by
\begin{equation}\label{eq:mcmcest}
\widehat Q^\mathrm{MC}_N := \frac{1}{N} \sum_{n=n_0+1}^{N+n_0} Q_{M,R}^{n} = \frac{1}{N} \sum_{n=n_0+1}^{N+n_0} \mathcal G\left(X(\theta^{n})\right),
\end{equation}
for any $n_0 \geq 0$, and only explicitly state the dependence on $n_0$ where needed.

\subsection{Convergence analysis of standard Metropolis-Hastings MCMC}
\label{sec:mcmc_conv}

Let us give a brief overview of the convergence properties of
  Algorithm 1, which we will need below in the analysis of the multilevel variant. For more details we refer the reader, e.g., to \cite{robert_casella}. 
Let
\begin{equation*}
K(\theta' | \theta) := \alpha^{M,R}(\theta' | \theta) \, q(\theta' | \theta) + \left(1- \int_{\mathbb R^R} \alpha^{M,R}(\theta'' | \theta) \, q(\theta'' | \theta) \dtheta''\right) \delta(\theta - \theta')
\end{equation*}
denote the transition kernel of the Markov chain $\{\theta^n\}_{n \in \mathbb N}$, 
{with $\delta(\cdot)$ the Dirac delta function}, and 
\begin{align*}
\mathcal E &= \{ \theta : \pi^{M,R}(\theta) > 0\}, \nonumber \\
\mathcal D &= \{ \theta : q(\theta | \theta^*) > 0 \text{ for some } \theta^* \in \mathcal E\}.
\end{align*}

The set $\mathcal E$ contains all parameter vectors which have a positive posterior probability, and is the set that Algorithm 1 should sample from. The set $\mathcal D$, on the other hand, consists of all samples which can be generated by the proposal distribution $q$, and hence contains the set that Algorithm~1 will actually sample from. For the algorithm to fully explore the target distribution, we therefore crucially require $\mathcal E \subset \mathcal D$. The following results are classical, and can be found in \cite{robert_casella}.

\begin{lemma}\label{lem:stat} Provided $\mathcal E \subset \mathcal D$, $\nu^{M,R}$ is a stationary distribution of the chain $\{\theta^n\}_{n \in \mathbb N}$.
\end{lemma}

Note that the condition $\mathcal E \subset \mathcal D$ is also sufficient for the transition kernel $K(\cdot | \cdot)$ to satisfy the usual detailed balance condition $K(\theta |  \theta^*) \, \pi^{M,R}(\theta^*) = K(\theta^* | \theta) \, \pi^{M,R}(\theta)$.

\begin{theorem}\label{thm:mcconv}
Suppose that $\EE_{\nu^{M,R}}\left[ |Q_{M,R}| \right] < \infty$ and 
\begin{equation}\label{eq:irreducible}
q(\theta | \theta^*) > 0, \text{ for all } (\theta, \theta^*) \in \mathcal E \times \mathcal E.
\end{equation}
Then  
\begin{equation*}
\lim_{N \rightarrow \infty} \widehat Q^\mathrm{MC}_N = \EE_{\nu^{M,R}}\left[ Q_{M,R} \right], \qquad \text{for any } \theta^0 \in \mathcal E \text{ and } n_0 \geq 0.
\end{equation*} 
\end{theorem}

The condition \eqref{eq:irreducible} is sufficient for the chain $\{\theta^n\}_{n \in \mathbb N}$ to be {\em irreducible}, and it is satisfied for example for the random walk sampler or for the pCN algorithm (cf. \cite{hsv11}). 
Lemma \ref{lem:stat} and Theorem \ref{thm:mcconv} above ensure that
asymptotically, sample averages computed with samples generated by
Algorithm 1 converge to the desired expected value. In particular, we
note that stationarity of $\{\theta^n\}_{n \in \mathbb N}$ is not
required in Theorem \ref{thm:mcconv}, and the estimator converges
for any burn--in $n_0 \geq 0$ and for all initial values $\theta^0 \in \mathcal E$.

Now that we have established the (asymptotic) convergence of the MCMC
estimator \eqref{eq:mcmcest}, let us bound its cost. 
We will quantify the accuracy of our estimator via the mean square error (MSE)
\begin{equation}\label{eq:mse}
e(\widehat Q^\mathrm{MC}_N)^2 := \mathbb{E}_\bTheta\left[\big(\widehat Q^\mathrm{MC}_N - \mathbb{E}_{{\rho}}(Q)\big)^2\right],
\end{equation}
where $\EE_\bTheta$ denotes the expected value with respect to the
joint distribution of $\bTheta := \{\theta^{n}\}_{n\in \mathbb{N}}$
  as generated by Algorithm~1 (not with respect to the target measure $\nu^{M,R}$).
We denote by $\mathcal C_\eps(\widehat Q^\mathrm{MC}_N)$ the
computational $\eps$-cost of the estimator, i.e. the number of
floating point operations needed to achieve a MSE $e(\widehat Q^\mathrm{MC}_N)^2 < \eps^2$.

Classically, the MSE can be written as the sum of the variance of the estimator and its bias squared,
\[
e(\widehat Q^\mathrm{MC}_N)^2 = \VV_\bTheta\left[\widehat Q^\mathrm{MC}_N\right] + \left( \EE_\bTheta\left[\widehat Q^\mathrm{MC}_N\right] - \EE_{\rho}\left[Q\right]\right)^2.
\]
Here, $\VV_\bTheta$ is again the variance with respect to the approximating measure generated by Algorithm~1.
Using the triangle inequality and linearity of expectation, we can
further bound this by
\begin{align}\label{eq:mse2}
e(\widehat Q^\mathrm{MC}_N)^2 
\leq \VV_\bTheta\left[\widehat Q^\mathrm{MC}_N\right] + {2} \left( \EE_\bTheta\left[\widehat Q^\mathrm{MC}_N\right] - \EE_{\nu^{M,R}}\left[\widehat Q^\mathrm{MC}_N \right]\right)^2 + {2} \left( \EE_{\nu^{M,R}}\left[Q_{M,R}\right] - \EE_{\rho}\left[Q \right]\right)^2 
\end{align}
The three terms in \eqref{eq:mse2} correspond to the three sources of
error in the MCMC estimator. The third (and last) term in
  \eqref{eq:mse2} is the discretisation error due to approximating $Q$
  by $Q_{M,R}$ and $\rho$ by $\nu^{M,R}$. The other two terms are
  the errors introduced by using an MCMC estimator for the expected
  value; the first term is the error due to using a finite number of samples and the second term is due to the samples not all being perfect (i.i.d.) samples from the target distribution $\nu^{M,R}$. 

Let us first consider the two MCMC related error terms. Quantifying, or even bounding, the variance and bias of an MCMC estimator in terms of the number of samples $N$ is not an easy task, and is in fact still a very active area of research. The main issue with bounding the variance is that the samples used in the MCMC estimator are not independent, which means that knowledge of the covariance structure is required in order to bound the variance of the estimator. Asymptotically, the behaviour of the MCMC related errors (i.e. Terms 1 and 2 on the right hand side of \eqref{eq:mse2}) can be described using the following Central Limit Theorem, which can again be found in \cite{robert_casella}.

Let $\tilde{\theta}^0 \sim \nu^{M,R}$. Then the auxiliary chain $\widetilde{\bTheta} := \{\tilde{\theta}^n\}_{n \in \mathbb N}$ constructed by Algorithm 1 starting from~$\tilde{\theta}^0$ is stationary, i.e. $\tilde \theta^n \sim \nu^{M,R}$ for all $n \geq 0$. The covariance structure of $\widetilde{\bTheta}$ is still implicitly defined by Algorithm 1 as for $\bTheta$. However, now $\VV_{\widetilde{\bTheta}}[\tilde{Q}^n_{M,R}] = \VV_{\nu^{M,R}}[\tilde{Q}_{M,R}]$, $\EE_{\widetilde{\bTheta}}[\tilde{Q}^n_{M,R}] = \EE_{\nu^{M,R}}[\tilde{Q}_{M,R}]$ and 
\[
\mathrm{Cov}_{{\widetilde{\bTheta}}}\left[\tilde{Q}_{M,R}^0, \, \tilde{Q}_{M,R}^n\right] = \EE_{\widetilde{\bTheta}} \left[\left( \tilde{Q}_{M,R}^0 - \EE_{\nu^{M,R}}[Q_{M,R}]\right) \left(\tilde{Q}_{M,R}^n - \EE_{\nu^{M,R}}[Q_{M,R}]\right)\right],
\] 
for any $n \geq 0$, where $\tilde{Q}^n_{M,R} :=
\mathcal{G}(X(\tilde{\theta}^n))$. The so-called {\em asymptotic
  variance} of the MCMC estimator is now defined as
\begin{equation}\label{def:assym_var}
\sigma_Q^2 := \VV_{\nu^{M,R}}\left[\tilde Q_{M,R}\right] + 2 \sum_{n=1}^\infty  \mathrm{Cov}_{\widetilde{\bTheta}}\left[\tilde{Q}_{M,R}^0, \, \tilde{Q}_{M,R}^n\right]. 
\end{equation}
Note that stationarity of the chain is assumed only in the definition of $\sigma_Q^2$, i.e. for $\widetilde{\bTheta}$, and it is not necessary for the samples
$\bTheta$ actually used in the computation of $\widehat Q^\mathrm{MC}_N$.

\begin{theorem}[Central Limit Theorem]\label{thm:clt} Suppose
  \eqref{eq:irreducible} holds, \ $\sigma_Q^2 < \infty$, \ and 
\begin{equation}\label{eq:aperiodic}
\PP\left[\alpha^{M,R}=1\right] < 1.
\end{equation}
Then 
we have, for any $n_0 \geq 0$ and $\theta^0 \in \mathcal E$,
\begin{equation*}
\sqrt{N}\left( \widehat Q^\mathrm{MC}_N - \EE_{\nu^{M,R}}\left[ Q_{M,R} \right]\right) \stackrel{D}{\longrightarrow} \mathcal N(0,\sigma_Q^2),
\end{equation*}
where $\stackrel{D}{\longrightarrow}$ denotes convergence in distribution.
\end{theorem}
The condition \eqref{eq:aperiodic} is sufficient for the chain $\bTheta$
to be {\em aperiodic}. It is difficult to prove 
theoretically. In practice, however, this condition is always satisfied, since not all proposals in Algorithm~1 will agree with the observed data and thus be accepted.
Theorem \ref{thm:clt} shows that asymptotically, the sampling
  error of the MCMC estimator decays at {the same rate as} the sampling error of an estimator based on i.i.d. samples.
Note that this includes both sampling errors, and so the constant $\sigma_{Q}^2$ is in general larger than in the i.i.d. case where it is simply $\VV_{\nu^{M,R}}\left[Q_{M,R}\right]$.

Since we are interested in a bound on the MSE of our MCMC estimator for a fixed number of samples $N$, we make the following assumption:\vspace{1.5ex}

\noindent
{\bf A1.} For any $N \in \mathbb N$,
\begin{equation}
\label{eq:A1}
\VV_\bTheta\left[\widehat Q^\mathrm{MC}_N\right] + \left( \EE_\bTheta\left[\widehat Q^\mathrm{MC}_N\right] - \EE_{\nu^{M,R}}\left[\widehat Q^\mathrm{MC}_N \right]\right)^2 \lesssim \frac{\VV_{\nu^{M,R}}[Q_{M,R}]}{N},
\end{equation} 
with a constant that is independent of $M$, $N$ and $R$.\vspace{1.5ex}

Such non-asymptotic bounds on the sampling errors are difficult to
  obtain, but have recently been proved for certain
  Metropolis--Hastings algorithms, see
  e.g. \cite{hsv11,rudolf_thesis,hss12}, provided the chain is
  sufficiently burnt--in. The implied 
constant in Assumption A1 usually depends on quantities such as
the covariances appearing in the asymptotic variance
$\sigma_Q^2$ and will in general only be independent of the dimension $R$ for judiciously chosen
  proposal distributions such as the pCN algorithm. For the simple random walk, for example, the hidden constant grows linearly in $R$. It is possible to relax Assumption A1 and prove convergence for algorithms also in this case, but we choose not to do this for ease of presentation.

To complete the error analysis, let us now consider the last term in
the MSE \eqref{eq:mse2}, the discretisation bias. As before, we assume 
$\EE_{\nu^{M,R}}\left[ Q_{M,R}\right] \rightarrow
\EE_{\rho}\left[Q \right]$ for $M, R \rightarrow \infty$ with a certain order of convergence
\begin{equation}\label{eq:meanconv}
\left|\EE_{\nu^{M,R}}\left[ Q_{M,R}\right] - \EE_{\rho}\left[Q \right] \right| \lesssim M^{-\alpha} + R^{-\alpha'},
\end{equation}
for some $\alpha, \alpha' > 0$. The rates $\alpha$ and $\alpha'$ will
be problem dependent.
Let now $R=M^{\alpha/\alpha'}$, such that the two error contributions
in \eqref{eq:meanconv} are balanced. Then it follows from
\eqref{eq:mse2}, \eqref{eq:A1} and \eqref{eq:meanconv} that the MSE of the MCMC estimator can be bounded by 
\begin{equation}\label{eq:msebound}
e(\widehat Q^\mathrm{MC}_N)^2 \lesssim \frac{\VV_{\nu^{M,R}}[Q_{M,R}]}{N} + M^{-\alpha}.
\end{equation}
Under the assumption that $\VV_{\nu^{M,R}}[Q_{M,R}]$ is approximately constant, independent of $M$ and $R$, it is hence sufficient to choose $N \gtrsim \eps^{-2}$ and $M \gtrsim \eps^{-1/\alpha}$ to get a MSE of $\mathcal{O}(\eps^2)$. 

To bound the computational cost to achieve this error, the so called {\em $\eps$-cost}, we
assume that one sample $Q_{M,R}^n$ can be obtained at cost $\mathcal
C(Q_{M,R}^n) \lesssim M^{\gamma}$, for some $\gamma > 0$. Thus, with $N \gtrsim \eps^{-2}$ and $M \gtrsim \eps^{-1/\alpha}$, the $\eps$--cost of our MCMC estimator can be bounded by
\begin{equation}
\label{eq:MCcost}
\mathcal C_\eps(\widehat Q^\mathrm{MC}_N) \lesssim NM^\gamma \lesssim \eps^{-2-\gamma/\alpha}.
\end{equation}

In many practical applications, especially in subsurface flow, both the
discretisation parameter $M$ and the length of the input
$R$ need to be very large in order for $\EE_{\nu^{M,R}}\left[
  Q_{M,R}\right]$ to be a good approximation to $\EE_{\rho}\left[
    Q\right]$. Moreover, 
as outlined, we need to use a large number of samples $N$ in order to get an accurate MCMC
estimator with a small MSE. Since each sample requires the evaluation
of the likelihood $\mathcal L_M(F_\mathrm{obs} | \theta^n)$, and this
is very expensive when $M$ and $R$ are large, the standard MCMC
estimator \eqref{eq:mcmcest} is often too expensive in practical
situations. Additionally, the acceptance rate of the algorithm can
be very low when $R$ is large. This means that the covariance between
the different samples will decay more slowly, which again makes the
hidden constant in Assumption A1 larger, and the number of samples we
have to take 
increases even further.

To overcome the prohibitively large computational cost of the standard MCMC estimator \eqref{eq:mcmcest}, we will now introduce a new multilevel version of the estimator.

\section{Multilevel Markov chain Monte Carlo algorithm}
\label{sec:mlmcmc}
The main idea of multilevel Monte Carlo (MLMC) simulation is very simple. We
sample not just from one approximation $Q_{M,R}$ of $Q$, but from several. Let
us recall the main ideas from \cite{giles08,cgst11}. 

Let $\{M_\ell\}_{\ell=0}^L \subset \mathbb{N}$ be an increasing sequence in $\mathbb{N}$, i.e. $M_0 < M_1 < \ldots < M_L =: M$, and assume for 
simplicity that there exists an $s \in \mathbb{N}\backslash\{1\}$ such that
\begin{equation}
\label{eq:growthm}
M_{\ell} = s \, M_{\ell-1}\,, \qquad \text{for all } \ell = 1,\ldots,L.
\end{equation}
We also choose a {(not necessarily strictly)} increasing sequence {$\{R_\ell\}_{\ell=0}^L \subset \mathbb{N}$, i.e. $R_\ell \ge R_{\ell-1}$, for all $\ell=1,\ldots,L$}. 
For each level $\ell$, denote correspondingly the parameter vector by $\theta_\ell \in \mathbb{R}^{R_\ell}$, the quantity of interest by $Q_\ell := Q_{M_\ell, R_\ell}$, the posterior distribution by $\nu^\ell := \nu^{M_\ell, R_\ell}$ and the posterior density by $\pi^\ell := \pi^{M_\ell, R_\ell}$. For simplicity we assume that the parameter vectors $\{\theta_\ell\}_{\ell=0}^L$ are nested, i.e. that $\theta_{\ell-1}$ is a subset of $\theta_\ell$, and that the elements of $\theta_\ell$ are independent.

As for multigrid methods applied to discretised (deterministic) PDEs,
the key is to avoid estimating the expected value of $Q_\ell$ directly on level $\ell$,
but instead to estimate the correction with respect to the next lower level. Since in the context of MCMC simulations, the target distribution $\nu^\ell$ depends on $\ell$, the new multilevel MCMC (MLMCMC) estimator has to be defined carefully. We will use the identity
\begin{equation}
\EE_{\nu^L}[Q_L] \ = \ \EE_{\nu^0}[Q_{0}] \;+ \;\sum_{\ell=1}^L \left(\EE_{\nu^\ell} [Q_\ell] - \EE_{\nu^{\ell-1}}[ Q_{\ell-1}]\right)
\label{eq:identity_l}
\end{equation}
as a basis. Note that in the case where the distributions are the same, the above reduces to the telescoping sum used for multilevel Monte Carlo estimators based on i.i.d samples.

The idea
is now to estimate each of the terms on the right hand side of \eqref{eq:identity_l} separately, in such a way that the variance of the resulting multilevel estimator is small. In particular, we will estimate each term in \eqref{eq:identity_l} by an MCMC estimator. The first term $\EE_{\nu^0}[Q_0]$ can be estimated using the standard MCMC estimator in Algorithm~1, i.e. $\widehat Q_{0, N_0}^\mathrm{MC}$ as in \eqref{eq:mcmcest} with $N_0$ samples. We need to be more careful in estimating the differences $\EE_{\nu^\ell} [Q_\ell] - \EE_{\nu^{\ell-1}}[ Q_{\ell-1}]$, and build an effective two-level version of Algorithm~1. {For every $\ell \geq 1$, we denote $Y_\ell := Q_\ell - Q_{\ell-1}$ and define the estimator on level $\ell$ as
\begin{equation}
\label{eq:levelest}
\widehat Y_{\ell, N_\ell}^{\mathrm{MC}} := \frac{1}{N_\ell} \sum_{n=n_0^\ell+1}^{n_0^\ell+N_\ell} Y_\ell^{n} = \frac{1}{N_\ell} \sum_{n=n_0^\ell+1}^{n_0^\ell+N_\ell} Q_\ell(\theta_\ell^n) - Q_{\ell-1}(\Theta_{\ell-1}^n),
\end{equation}
where $n_0^\ell$ again denotes the burn-in of the estimator, $N_\ell$ is the 
number of samples on level $\ell$ and $\Theta_{\ell-1}$ has the same dimension as $\theta_{\ell-1}$. 
The main ingredient in this two--level estimator is a judicious choice of 
the two Markov chains $\{\theta_\ell^n\}$ and $\{\Theta_{\ell-1}^n\}$ 
(see Section \ref{sec:twolevel}).
The full MLMCMC estimator is defined as\vspace{-1ex}
\begin{equation}\label{eq:mlmcmcest}
\widehat Q^\mathrm{ML}_{L,\{N_\ell\}} \ := \ \widehat Q_{0, N_0}^\mathrm{MC} \;+\; \sum_{\ell=1}^L \widehat Y_{\ell,N_\ell}^\mathrm{MC},
\end{equation}
where it is important that the two chains $\{\theta^n_\ell\}_{n \in
  \mathbb{N}}$ and $\{\Theta^n_\ell\}_{n \in \mathbb{N}}$, that are
used in $\widehat Y_{\ell,N_\ell}^\mathrm{MC}$ and in $\widehat
Y_{\ell+1,N_{\ell+1}}^\mathrm{MC}$ respectively, are drawn from the
same posterior distribution $\nu^\ell$, so that $\widehat
Q^\mathrm{ML}_{L,\{N_\ell\}}$ is an unbiased estimator of
$\EE_{\nu^L}[Q_L]$.

There are two main ideas {in \cite{giles08,cgst11}} underlying the reduction in computational cost associated with the multilevel estimator. Firstly, samples of $Q_\ell$, for $\ell < L$, are cheaper to compute than samples of~$Q_L$, reducing the cost of the estimators on the coarser levels for any fixed number of samples.
Secondly, if the variance of $Y_\ell = Q_\ell(\theta_\ell) - Q_{\ell-1}(\Theta_{\ell-1})$ tends to 0 as $\ell \rightarrow \infty$, we need only a small number of samples to obtain a sufficiently accurate estimate of the expected value of $Y_\ell$ on the fine grids, and so the computational effort on the fine grids is also greatly reduced. 

By using the telescoping sum \eqref{eq:identity_l} 
and by sampling from the posterior distribution $\nu^\ell$ on level $\ell$, we ensure that a sample of $Q_\ell$, for $\ell < L$, is indeed cheaper to compute than
a sample of $Q_L$. It remains to ensure that the variance of $Y_\ell = Q_\ell(\theta_\ell) - Q_{\ell-1}(\Theta_{\ell-1})$ tends to 0 as $\ell \rightarrow \infty$. This will be ensured by the choice of $\theta_\ell$ and $\Theta_{\ell-1}$. Note that crucially, this requires the two chains $\{\theta_\ell^n\}$ and $\{\Theta_{\ell-1}^n\}$ to be correlated. However, as long as the stationary marginal distributions of $\{\theta_\ell^n\}$ and $\{\Theta_{\ell-1}^n\}$ are $\nu^\ell$ and $\nu^{\ell-1}$ respectively, this correlation does not introduce any bias in the telescoping sum \eqref{eq:identity_l}.

\subsection{The estimator for $Q_\ell - Q_{\ell-1}$}
\label{sec:twolevel}

Let us fix $1 \le \ell \le L$.
 The challenge is now to generate the chains $\{\theta_\ell^n\}_{n \in \mathbb N}$ and $\{\Theta_{\ell-1}^n\}_{n \in \mathbb N}$ such that the variance of $Y_\ell$ is small. To this end, we partition the chain $\theta_\ell$ into two parts: the entries which are present already on level $\ell-1$ (the ``coarse'' modes), and the new entries on level $\ell$ (the ``fine'' modes): 
\begin{equation*}
\theta_\ell = [\theta_{\ell,C} \, , \, \theta_{\ell, F}],
\end{equation*}
where $\theta_{\ell,C}$ has length $R_{\ell-1}$, i.e. the same length as $\Theta_{\ell-1}$. The vector $\theta_{\ell,F}$ has length $R_\ell - R_{\ell-1}$.

An easy way to construct $\theta_\ell^n$ and $\Theta_{\ell-1}^n$ 
such that the variance of $Y_\ell$ is small, 
would be to generate $\theta_\ell^n$ first, and then simply use
$\Theta_{\ell-1}^n = \theta_{\ell,C}^n$. However, since we require
$\Theta_{\ell-1}^n$ to come from a Markov chain with stationary
distribution $\nu^{\ell-1}$, and $\theta_{\ell}^n$ comes from the
distribution $\nu^{\ell}$, this approach would lead to additional
bias. 
We do, however, use a similar idea in Algorithm 2.

\begin{figure}[ht]
\begin{framed}
\noindent
{\bf ALGORITHM 2. (Metropolis Hastings MCMC for $Q_{\ell} - Q_{\ell-1}$)} \vspace{2ex} \\ 
Choose initial states $\Theta_{\ell-1}^0 \sim \nu^{\ell-1}$ and $\theta_{\ell}^0 := [\Theta_{\ell-1}^0 \,,\,\theta_{\ell,F}^0]$. For $n \geq 0$:
\begin{itemize} 
\item {\bf On level $\ell-1$:}
Generate an independent sample $\Theta_{\ell-1}^{n+1}$ from the distribution $\nu^{\ell-1}$.
\item {\bf On level $\ell$:}
Given $\theta_{\ell}^n$ and $\Theta_{\ell-1}^{n+1}$, generate
$\theta_{\ell}^{n+1}$ using Algorithm 1 with the specific proposal
distribution $q^{\ell}_{\mathrm{ML}}(\theta_{\ell}' \, | \,
\theta_{\ell}^n)$ induced by taking $\theta_{\ell,C}' :=
\Theta_{\ell-1}^{n+1}$ and by generating a proposal for
$\theta_{\ell,F}'$ from some proposal distribution
$q_{\mathrm{ML}}^{\ell,F}(\theta_{\ell,F}'\, | \,\theta_{\ell,F}^n)$
that is independent of the coarse modes. The acceptance probability is\vspace{-1.5ex}
$$
\alpha^{\ell}_{\mathrm{ML}}(\theta_{\ell}' \, | \, \theta_{\ell}^n) \;=\; 
\min \left\{ 1, \frac{\pi^{\ell}(\theta_{\ell}') \, q^{\ell}_{\mathrm{ML}}(\theta_{\ell}^n | \theta_{\ell}')}{\pi^{\ell}(\theta_{\ell}^n)\, q^{\ell}_{\mathrm{ML}}(\theta_{\ell}'  |  \theta_{\ell}^n)}\right\}. \vspace{-3ex}
$$

\end{itemize}
\end{framed}
\end{figure}

Let us for the moment assume that we have a way of producing
  i.i.d. samples from the posterior distribution $\nu^{\ell-1}$. Since
  the distributions $\nu^{\ell-1}$ and $\nu^\ell$ are both
  approximations of the true posterior distribution $\rho$, and differ
  only in the choice of approximation parameters $M$ and $R$, the
  distributions $\nu^{\ell-1}$ and $\nu^\ell$ will, for sufficiently
  large $\ell$, be very similar. The distribution $\nu^{\ell-1}$ is
  hence an ideal candidate for the proposal distribution on level
  $\ell$, and this is what is used in Algorithm~2. First, we generate
  a sample $\Theta_{\ell-1}^{n+1}$ from the distribution
  $\nu^{\ell-1}$, which is independent of the previous sample
  $\Theta_{\ell-1}^{n}$. We will use the independence of these samples
  in Lemma \ref{lem:all}. Based on $\Theta_{\ell-1}^{n+1}$, we then
  generate $\theta_\ell^{n+1}$ using a new two-level proposal density
  $q_{\mathrm{ML}}^{\ell}$ in conjunction with the usual
  Metropolis-Hastings accept/reject step in Algorithm~1. In
  particular, to make a proposal on level $\ell$, we take
  $\theta_{\ell,C}'=\Theta_{\ell-1}^{n+1}$ and independently generate
  $\theta_{\ell,F}'$ from a proposal distribution
  $q_{\mathrm{ML}}^{\ell,F}$ for the fine modes, which can again be a simple random walk or the pCN algorithm.

At each step in Algorithm 2, there are two different outcomes, depending on whether we accept or reject on level~$\ell$. The different possibilities  
are given in Table \ref{tab:outcomes}. Observe that when we accept on level~$\ell$, we have $\theta_{\ell,C}^{n+1} = \Theta_{\ell-1}^{n+1}$, i.e. the coarse modes are the same. If, on the other hand, we reject on level $\ell$, we crucially return to the previous state $\theta_\ell^n$ on that level, which means that the coarse modes of the two states may differ.
\begin{table}[h]\label{tab:twolev}
\begin{center}\vspace{1ex}
\begin{tabular}{|c|c|c|}
\hline
 Level $\ell$ test & $\Theta_{\ell-1}^{n+1}$ & $\theta_{\ell,C}^{n+1}$ \\\hline
accept & $\Theta_{\ell-1}^{n+1}$ & $\Theta_{\ell-1}^{n+1}$   \\ \hline
reject & $\Theta_{\ell-1}^{n+1}$  & $\theta_{\ell,C}^n$   \\ \hline
\end{tabular}
\end{center}
\caption{\label{tab:outcomes}Possible states of $\Theta_{\ell-1}^{n+1}$ and $\theta_{\ell,C}^{n+1}$ in Algorithm 2.}
\end{table}

In general, this ``divergence'' of the coarse modes may mean that the
variance of $Y_\ell$ does not go to $0$ as $\ell \to \infty$ for a
particular application. But provided the modes are ordered according
to their relative ``influence'' on the likelihood
$\mathcal{L}(F_\text{obs} | \theta)$, we can guarantee that
$\alpha^\ell_{\mathrm{ML}}(\theta_{\ell}' | \theta_{\ell}^n) \to 1$ and thus that the variance of $Y_\ell$ does in fact tend to 0 as $\ell \to \infty$. We will show this for a subsurface flow application in Section \ref{sec:mod}.

The specific proposal distribution $q^\ell_{\mathrm{ML}}$ in 
Algorithm 2 can be computed very easily and at no additional cost,
leading to a simple formula for the ``two-level'' acceptance
probability~$\alpha^\ell_{\mathrm{ML}}$.

\begin{lemma}\label{lem:all}  Let $\ell \ge 1$. Then
\[
{\alpha^{\ell}_{\mathrm{ML}}(\theta_{\ell}' \, | \, \theta_{\ell}^n)} \;=\; 
\min \left\{ 1, \frac{\pi^{\ell}(\theta_{\ell}') \,  {\pi^{\ell-1}(\theta_{\ell,C}^{n})\, q_{\mathrm{ML}}^{\ell,F}(\theta_{\ell,F}^n | \theta_{\ell,F}')}}{\pi^{\ell}(\theta_{\ell}^n) \,  {\pi^{\ell-1}(\theta_{\ell,C}')\, q_{\mathrm{ML}}^{\ell,F}(\theta_{\ell,F}' | \theta_{\ell,F}^n)}}\right\}
\]
and the induced transition kernel $K^{\ell}_{\mathrm{ML}}$ satisfies detailed balance.\pagebreak

Furthermore, if the distribution $q_{\mathrm{ML}}^{\ell,F}$ is either (i) symmetric, or (ii) the pCN proposal distribution, then\vspace{-1ex}
\[
\alpha^{\ell}_\mathrm{ML}(\theta_\ell' \, | \, \theta_\ell^n) \;=\; 
\left\{ \begin{array}{ll}
\displaystyle \min \left\{ 1, \frac{\pi^\ell(\theta_\ell') \,
    {\pi^{\ell-1}(\theta_{\ell,C}^{n})}}{\pi^\ell(\theta_\ell^n) \,
    {\pi^{\ell-1}(\theta_{\ell,C}')}}\right\}, & \text{Case (i)},\\
 & \\[-1.5ex]
{\displaystyle \min \left\{ 1, \frac{\mathcal L_\ell(F_\mathrm{obs}\, |\,
    \theta_\ell') \, {\mathcal L_{\ell-1}(F_\mathrm{obs}\, |\,
    \theta_{\ell,C}^n)}}{\mathcal
    L_\ell(F_\mathrm{obs}\, |\, \theta_\ell^n) \, {\mathcal
    L_{\ell-1}(F_\mathrm{obs}\, |\, \theta_{\ell,C}')}} \right\},}
\ \ & \text{Case (ii).}
\end{array} \right.
\]
\end{lemma}
\begin{proof}
Since the proposals for the coarse modes $\theta_{\ell,C}$ and for the
fine modes $\theta_{\ell,F}$ are generated independently, the proposal
density $q^\ell_\mathrm{ML}(\theta_\ell' \, | \theta_\ell^n)$ can be written as a product of densities on the two parts of $\theta_\ell$, {i.e. $q^{\ell,C}_\mathrm{ML}$ and $q^{\ell,F}_\mathrm{ML}$}. For the coarse part of the proposal distribution, we simply have $q^{\ell,C}_\mathrm{ML}(\theta_{\ell,C}' |  \theta_{\ell,C}^n) = \pi^{\ell-1}(\theta_{\ell,C}')$ and $q^{\ell,C}_\mathrm{ML}(\theta_{\ell,C}^n |  \theta_{\ell,C}') = \pi^{\ell-1} (\theta_{\ell,C}^n)$.

This completes the proof of the first result. Detailed balance for $K^{\ell}_\mathrm{ML}$ follows trivially due to the Metropolis-Hastings construction.
The corollary for symmetric distributions $q_\mathrm{ML}^{\ell,F}$ follows by definition. The corollary for pCN proposals follows from the identity $q_\mathrm{ML}^{\ell,F}(\theta_{\ell,F}^n | \theta_{\ell,F}') / q_\mathrm{ML}^{\ell,F}(\theta_{\ell,F}' | \theta_{\ell,F}^n) = \pi_0^{\ell,F}(\theta_{\ell,F}^n) / \pi_0^{\ell,F}(\theta_{\ell,F}')$ (see, e.g. \cite{cds12}), together with the factorisation $\pi_0^\ell(\theta_\ell) = \pi_0^{\ell-1}(\theta_{\ell,C}) \, \pi_0^{\ell,F}(\theta_{\ell,F})$.
\end{proof}

\subsection{Recursive sub-sampling to generate i.i.d. samples from $\nu^{\ell-1}$}

In practice, it will not be possible to generate independent samples
of the coarse level posterior distribution $\nu^{\ell-1}$ directly. We
instead suggest approximating independent samples of $\nu^{\ell-1}$
using Algorithm~1 in the following manner: After a sufficiently long
burn-in period, Algorithm~1 will produce samples which are
(approximately) distributed according to $\nu^{\ell-1}$. Although the
samples produced in this way are correlated, the correlation between
the $n$th and $(n+j)$th sample decays as $j$ increases, and for
sufficiently large $j$, the samples $\Theta_{\ell-1}^n$ and
$\Theta_{\ell-1}^{n+j}$ will be nearly uncorrelated. Hence, an i.i.d
sequence of samples of $\nu^{\ell-1}$ can be approximated by
subsampling a chain $\{\Theta_{\ell-1}^n\}_{n \in
  \mathbb N}$ generated by Algorithm 1 with, e.g., the pCN proposal
distribution.

This procedure can be applied very  naturally in a recursive
manner. Starting on  the coarsest level, burning in a Markov chain of samples
and subsampling this chain to produce (nearly) independent samples
from $\nu^0$ we can then apply Algorithm~2 to produce a Markov chain
of samples from $\nu^1$. This can then be subsampled again to apply
Algorithm~2 on level 2. Continuing in this way, we can recursively
produce independent samples from $\nu^{\ell-1}$ for any $\ell >
0$. See Algorithm~3 in Section \ref{sec:num} for details.

Although, in general the i.i.d. samples of $\nu^{\ell-1}$ will in practice have to be
approximated,  for the analysis of our multilevel algorithm we will
assume that the chains $\{\Theta_{\ell-1}^n\}_{n \in \mathbb N}$ and $\{\theta_{\ell}^n\}_{n
  \in \mathbb N}$ are generated as in Algorithm 2.
The additional bias introduced in the practical Algorithm 3 below is
in fact so small that we did initially not detect it in our numerical
experiments, even for very short subsampling rates.

\subsection{Convergence analysis of the multilevel MCMC estimator}
\label{sec:mlmcmc_conv}
Let us now move on to convergence properties of the multilevel
estimator. As in Section \ref{sec:mcmc_conv}, we define, for all $\ell=0,\dots,L$, the sets
\begin{align*}
\mathcal E^\ell &= \{ \theta_\ell : \pi^{\ell}(\theta_\ell) > 0\}, \nonumber \\
\mathcal D^\ell &= {\{ \theta_\ell : q^{\ell}_\mathrm{ML}(\theta_\ell \, | \,\theta^*_\ell) > 0 \text{ for some } \theta_\ell^* \in \mathcal E^\ell\}}.
\end{align*}

The following convergence results follow from the classical results, due to the telescoping sum property \eqref{eq:identity_l} and the algebra of limits.

\begin{lemma}\label{lem:stat_l} Provided $\mathcal E^\ell \subset \mathcal D^\ell$, $\nu^\ell$ is a stationary marginal distribution of the chain $\{\theta^n_\ell\}_{n \in \mathbb N}$.
\end{lemma}

\begin{theorem}\label{thm:mlmcconv}
Suppose {that for all $\ell=0, \dots, L$,} $\EE_{\nu^\ell}\left[ |Q_\ell| \right] < \infty$ and 
\begin{equation}\label{eq:mlirreducible}
q^\ell_\mathrm{ML}(\theta_\ell \, | \, \theta_\ell^*) > 0, \quad \text{for all } \ \theta_{\ell},
  \theta_\ell^* \in \mathcal E^\ell.
\end{equation}
Then
\begin{equation*}
\lim_{\{N_\ell\} \rightarrow \infty}  \widehat Q^\mathrm{ML}_{L, \{N_\ell\}} = \EE_{\nu^L}\left[ Q_L \right], \qquad \text{for any } \ \theta_\ell^0 \in \mathcal E^\ell \ \ \text{and} \ \ n_0^\ell \geq 0.
\end{equation*} 
\end{theorem}

Let us have a closer look at the irreducibility condition
\eqref{eq:mlirreducible}. As in the proof of Lemma \ref{lem:all}, we have 
\[
q^\ell_\mathrm{ML}(\theta_\ell | \theta^*_\ell) = \pi^{\ell-1}( \theta_{\ell,C} ) \, q^{\ell,F}_\mathrm{ML}(\theta_{\ell,F}  |  \theta_{\ell,F}^*)
\]
and thus
  \eqref{eq:mlirreducible} holds, if and only if 
  $\pi^{\ell-1}( \theta_{\ell,C} )$
  and $q^{\ell,F}_\mathrm{ML} (\theta_{\ell,F} | \theta_{\ell,F}^*)$ are both
  positive, for all
  $(\theta_\ell, \theta_\ell^*) \in \mathcal E^\ell \times \mathcal
  E^\ell$. Both terms are positive for common choices of likelihood, prior 
  and proposal distributions.

We finish the abstract discussion of the new, hierarchical multilevel Metropolis-Hastings MCMC algorithm with the main theorem that establishes a bound on the $\varepsilon$-cost of the multilevel estimator under certain assumptions on the MCMC error, on the (weak) model error, on the strong error between the states on level $\ell$ and on level $\ell-1$ (in the two-level estimator for $Y_\ell$), as well as on the cost $\mathcal{C}_\ell$ to advance Algorithm 2 by one state from $n$ to $n+1$ (i.e.~one evaluation of the likelihood on level $\ell$ and one on level $\ell-1$). As in the case of the standard MCMC estimator, this bound is obtained by quantifying and balancing the decay of the bias and the sampling errors of the estimator. 

To state our assumption on the MCMC error and to define the mean
square error of the estimator, we introduce the following notation. We
define $\bTheta_\ell := \{\theta_\ell^n\}_{n\in\mathbb{N}} \cup
\{\Theta_{\ell-1}^n\}_{n\in\mathbb{N}}$, for $\ell \ge 1$, and
$\bTheta_0 := \{\theta_0^n\}_{n\in\mathbb{N}}$\,, and define by
$\EE_{\bTheta_\ell}$ (respectively $\VV_{\bTheta_\ell}$) the expected
value (respectively variance) with respect to the distribution of
$\bTheta_\ell$ generated by Algorithm 2. Furthermore, let us
  denote by $\nu^{\ell,\ell-1}$ the joint distribution of
  $\theta_\ell$ and $\Theta_{\ell-1}$, for $\ell \geq 1$, which is
  defined by the marginals of $\theta_\ell$ and $\Theta_{\ell-1}$
  being $\nu^\ell$ and $\nu^{\ell-1}$, respectively, and the
  correlation being determined by Algorithm 2. For convenience, we
  define $Y_0 := Q_0$, $\nu^{0,-1} := \nu^0$ and $M_{-1}=R_{-1} = 1$.

\begin{theorem}\label{thm:main}
Let $\eps < \exp[-1]$  and suppose
there are positive constants $\alpha, \alpha', \beta, \beta', \gamma  > 0$ such that 
$\alpha \geq \fracs{1}{2}\,\min(\beta,\gamma)$.
Under the following assumptions, for $\ell=0,\ldots,L$, 
\begin{itemize}
\item[{\bf M1.}]
$\displaystyle
\left| \EE_{\nu^\ell}[Q_{\ell}] - \EE_{\rho}[Q] \right| \ \leq C_\mathrm{M1} \left(\ M_\ell^{-\alpha} + R_\ell^{-\alpha'}\right)
$
\item[{\bf M2.}]
$\displaystyle
\VV_{\nu^{\ell,\ell-1}}[Y_\ell] \ \leq C_\mathrm{M2} \ \left({M_{\ell-1}^{-\beta}+R_{\ell-1}^{-\beta'}}\right)
$
\item[{\bf M3.}]
$\displaystyle
\VV_{\bTheta_\ell}[\widehat Y_{\ell,N_\ell}^\mathrm{MC}] + \left(\EE_{\bTheta_\ell}[\widehat Y_{\ell,N_\ell}^\mathrm{MC}] - \EE_{\nu^{\ell,\ell-1}}[\widehat Y_{\ell,N_\ell}^\mathrm{MC}]\right)^2\ \leq C_\mathrm{M3}  \ N_\ell^{-1} \, \VV_{\nu^{\ell,\ell-1}}[Y_\ell]
$
\item[{\bf M4.}]
$\displaystyle
\mathcal{C}_{\ell} \ \leq C_\mathrm{M4} \ M_{\ell}^{\gamma},
$
\end{itemize}
and provided $R_\ell \gtrsim M_\ell^{\max\{\alpha/\alpha',
    \beta/\beta' \}}$, there exists a number of levels $L$ and a sequence 
$\{N_{\ell}\}_{\ell=0}^L$ such that\vspace{-1ex}
\[
e(\widehat{Q}^\mathrm{ML}_{L, \{N_\ell\}})^2 \ := \ 
\EE_{\cup_\ell\bTheta_\ell}\left[ \big(\widehat{Q}^\mathrm{ML}_{L, \{N_\ell\}}- \EE_{\rho}[Q]\big)^2\right] < \eps^2,
\]
and
\[
\mathcal{C}_\eps(\widehat{Q}^\mathrm{ML}_{L, \{N_\ell\}}) \ \leq C_\mathrm{ML} \ \left\{\begin{array}{ll}
\eps^{-2} \, |\log \eps|             ,    & \text{if } \ \beta>\gamma, \\[0.1in]
\eps^{-2} \, |\log \eps|^3,    & \text{if } \ \beta=\gamma, \\[0.1in]
\eps^{-2-(\gamma\!-\!\beta)/\alpha} \, |\log \eps| , & \text{if } \ \beta<\gamma.
\end{array}\right.
\]
\end{theorem}
\begin{proof} 
The proof of this theorem is very similar to the proof of the complexity theorem in the case of multilevel estimators based on i.i.d samples (cf.~\cite[Theorem 1]{cgst11}), which can be found in the appendix of \cite{cgst11}. First note that by assumption we have $R_\ell^{-\alpha'} \lesssim M_\ell^{-\alpha}$ and $R_\ell^{-\beta'} \lesssim M_\ell^{-\beta}$.\pagebreak

Furthermore, in the same way as in \eqref{eq:mse2}, we can expand  
\[
e(\widehat{Q}^\mathrm{ML}_{L, \{N_\ell\}})^2 \ \leq \ \VV_{\cup_\ell\bTheta_\ell}\left[\widehat{Q}^\mathrm{ML}_{L, \{N_\ell\}}\right] + 2 \underbrace{\left(\EE_{\cup_\ell\bTheta_\ell}\left[\widehat{Q}^\mathrm{ML}_{L, \{N_\ell\}}\right]-\EE_{\nu^L}\left[\widehat{Q}^\mathrm{ML}_{L, \{N_\ell\}}\right]\right)^2}_{(I)} + 2 \Big(\EE_{\nu^L}[Q_L] - \EE_{\rho}[Q]\Big)^2.
\]
It follows from the Cauchy Schwarz inequality that 
\[
\VV_{\cup_\ell\bTheta_\ell}\left[\widehat{Q}^\mathrm{ML}_{L, \{N_\ell\}}\right] = \sum_{l=0}^L \VV_{\bTheta_\ell}[\widehat Y_{\ell,N_\ell}^\mathrm{MC}] + 2 \sum_{0 \leq \ell < \ell' \leq L}\textrm{Cov}_{\cup_\ell\bTheta_\ell}[\widehat Y_{\ell,N_\ell}^\mathrm{MC},\widehat Y_{\ell',N_{\ell'}}^\mathrm{MC}]  \lesssim (L+1) \sum_{l=0}^L \VV_{\bTheta_\ell}[\widehat Y_{\ell,N_\ell}^\mathrm{MC}]  .
\]
We can bound the second term in the MSE above by
\[
(I) \;  = \; \bigg( \sum_{l=0}^L \left( \EE_{\bTheta_\ell}\left[\widehat Y_{\ell,N_\ell}^\mathrm{MC}\right] - \EE_{\nu^{\ell,\ell-1}}\left[\widehat Y_{\ell,N_\ell}^\mathrm{MC}\right]\right)\bigg)^2 
\; \leq \; (L+1) \sum_{l=1}^L \left( \EE_{\bTheta_\ell}\left[\widehat Y_{\ell,N_\ell}^\mathrm{MC}\right] - \EE_{\nu^{\ell,\ell-1}}\left[\widehat Y_{\ell,N_\ell}^\mathrm{MC}\right]\right)^2,
\]
and thus it follows from Assumption M3  that
\begin{equation}\label{eq:mse_multi}
e(\widehat{Q}^\mathrm{ML}_{L, \{N_\ell\}})^2 \lesssim (L+1) \sum_{\ell=0}^L N_\ell^{-1} \, \VV_{\nu^{\ell,\ell-1}}[Y_\ell] +  \Big(\EE_{\nu^L}[Q_{L}] - \EE_{\rho}[Q] \Big)^2. 
\end{equation}
In contrast to 
i.i.d case, we have an additional factor $(L+1)$ multiplying the
sampling error term on the right hand side of
\eqref{eq:mse_multi}. Hence, in order to make this term less than
$\eps^2/2$, the number of samples $N_\ell$ needs to be increased by a
factor of $(L+1)$ compared to the i.i.d. case, which  also increases
the cost of the multilevel estimator by a factor of $(L+1)$. The remainder of the proof remains identical.

Since $L$ is chosen such that the second term in \eqref{eq:mse_multi} (the bias of the multilevel estimator) is less than $\eps^2/2$, it follows from Assumption M1 that $L+1 \lesssim |\log \eps|$. The bounds on the $\eps$-cost 
then follow as in \cite[Theorem 1]{cgst11}, but with an extra $|\log \eps|$ factor.
\end{proof}

Note that in our proof we do not require the estimators  $\widehat
  Y_{\ell,N_\ell}^\mathrm{MC}$, $\ell=0,\ldots,L$, to be
  independent. However, in practice we found that independent
  estimators lead to a faster absolute performance of the multilevel
  estimator (in terms of cost
  versus error).
 
Assumptions M1 and M4 are the same assumptions as in the single--level case, and are related to the {bias in the model (e.g. due to discretisation) and to} the cost per sample, respectively. Assumption M3 is similar to assumption A1, in that it is a non-asymptotic bound for the sampling errors of the MCMC estimator $\widehat Y_{\ell,N_\ell}^\mathrm{MC}$. For this assumption to hold, it is in general necessary that the chains have been sufficiently burnt in, i.e. that {the values $n_0^\ell$ are} sufficiently large.

\section{Model Problem}
\label{sec:mod}
In this 
section, we will apply the proposed MLMCMC algorithm to a simple model problem arising in subsurface flow modelling. Probabilistic uncertainty quantification in subsurface flow is of interest in a 
number of situations, as for example in risk analysis for radioactive waste 
disposal or in oil reservoir simulation. The classical equations governing 
(steady state) single--phase subsurface flow consist of Darcy's law coupled with an 
incompressibility condition (see e.g. \cite{demarsily,cgss00}):
\begin{equation}\label{eq:mod}
w + k \nabla p = g \quad \text{and} \quad 
\text{div } w = 0, \quad
\text{in } \ D  \subset \mathbb{R}^d, \ d=1,2,3, 
\end{equation}  
subject to suitable 
boundary conditions. In physical terms, $p$ denotes the 
pressure head of the fluid, $k$ is the permeability tensor, 
$w$ is the filtration 
velocity (or Darcy flux) and $g$ is the source term. 

\subsection{Uncertainty quantification}
A typical approach to quantify uncertainty in $p$ and $w$ is to model 
the permeability as a random field $k = k(x,\omega)$ on 
$D \times \Omega$, for some probability space $(\Omega, \mathcal A, \mathbb P)$. The mean and covariance structure of $k$ has to be 
inferred from the (limited) geological information available. This means that \eqref{eq:mod} becomes a system of PDEs 
with random coefficients, which can be written in second order form as
\begin{align}\label{eq:mod2}
-\nabla \cdot (k(x, \omega) \nabla p(x, \omega)) &= 
f(x), \qquad \mathrm{in}\quad D, 
\end{align}
with $f := - \text{div } g$.
This means that the solution $p$ itself will also be a random field on 
$D \times \Omega$. For simplicity, we shall restrict ourselves to Dirichlet conditions $p(\omega,x) = \psi(x)$ on $\partial D$, and assume that the boundary data $\psi$ and the 
source term $g$ are known (and thus deterministic). 

In this general form solving \eqref{eq:mod2} is extremely challenging 
computationally, and so in practice it is common to use relatively simple 
models for $k$ that are as faithful as possible to the 
measurements. One model that has been studied extensively is a log-normal 
distribution for $k$, i.e.~replacing the permeability 
tensor by a scalar valued field whose $\log$ is Gaussian. It 
guarantees that $k > 0$ almost surely (a.s.) in~$\Omega$,
and it allows the permeability to
vary over many orders of magnitude, which is typically the case.

When modelling a whole oil reservoir or a sufficiently
large region around a potential radioactive waste repository,
the correlation length scale for $k$ is typically significantly
smaller than the size of the computational region. 
In addition, typical sedimentation 
processes lead to fairly irregular structures and pore networks. Faithful
models should therefore also only assume limited spatial regularity of $k$.
A covariance function that has been proposed in the application literature
(cf. \cite{hk85}) 
is the following exponential two-point covariance function for $\log k$:
\begin{equation}
\label{eq:cov}
C(x,y) \ := \ \sigma^2 
\mathrm{exp}\left(-\frac{\|x - y\|_r}{\lambda}\right), \qquad x,y \in D,
\end{equation}
where $\|\cdot\|_r$ denotes the $\ell_r$-norm in $\mathbb{R}^d$ and typically 
$r=1$ or $2$. The parameters $\sigma^2$ and $\lambda$ denote
{\em variance} and {\em correlation length}, respectively. In 
subsurface flow applications typically only $\sigma^2 \ge 1$ and 
$\lambda \le \mathrm{diam}\, D$ will be of interest. The choice of
covariance function in \eqref{eq:cov} implies that $k$ is {\em homogeneous} and it follows from Kolmogorov's theorem \cite{daprato_zabczyk} that $k(\cdot,\omega) \in C^{0,t}(D)$ a.s., for any $t < 1/2$. 

For the purpose of this paper, we will assume that $k$ is a log-normal random field, where $\log k$ has mean zero and exponential covariance function \eqref{eq:cov} with $r=1$. However, other models for $k$ are possible, and the required theoretical results can be found in \cite{cst11,tsgu13,teckentrup12}. 

Let us now put model problem \eqref{eq:mod2} into context for the MCMC
and MLMCMC methods described in sections \ref{sec:mcmc} and
\ref{sec:mlmcmc}. The quantity of interest $Q$ is in this case some
functional $\mathcal G$ of the PDE solution $p$, and $Q_{M,R}$ is the
same functional $\mathcal G$ evaluated at a discretised solution
$p_{M,R}$. The discretisation level $M$ denotes the number of degrees
of freedom for the numerical solution of \eqref{eq:mod2} for a given
sample and the parameter $R$ denotes the number of random variables used to
model the permeability $k$. The random vector $X$ will contain the $M$
degrees of freedom of the discrete pressure $p_{M,R}$.

For the spatial discretisation of model problem \eqref{eq:mod2}, we
will use standard, continuous, piecewise linear finite elements
(FEs), see e.g. \cite{brenner_scott,ciarlet} for more details. Other spatial discretisation schemes are possible, see for example \cite{cgst11} for a numerical study with finite volume methods and \cite{gsu12} for a theoretical treatment of mixed finite elements. We choose a regular triangulation $\mathcal T_h$ of mesh width $h$ of our spatial domain $D$, which results in $M = \mathcal{O}(h^{-d})$ degrees of freedom for the numerical approximation. 

In order to apply the proposed MCMC methods to model problem
  \eqref{eq:mod2}, we need to represent the permeability $k$ in terms
  of a set of random variables. For this, we will use the Karhunen-Lo\`eve (KL-) expansion. For the Gaussian field $\log k$, this is an expansion in terms of 
a countable set of independent, standard Gaussian random variables 
$\{\xi_n\}_{n \in \mathbb{N}}$. It is given by
\begin{equation*}
\log k(\omega,x)=  \sum_{n=1}^\infty \sqrt{\mu_n}\phi_n(x)\xi_n(\omega),
\end{equation*}
where $\{\mu_n\}_{n \in \mathbb{N}}$ are the eigenvalues and $\{\phi_n\}_{n \in \mathbb{N}}$ 
the corresponding $L^2$-normalised eigenfunctions of the covariance operator with kernel 
function $C(x,y)$. For more details on its derivation and  
properties, see e.g.~\cite{ghanem_spanos}. We will here only mention 
that the eigenvalues $\{\mu_n\}_{n \in \mathbb{N}}$ are non--negative with 
$\sum_{n\geq 0}\mu_n<\infty.$ For the particular covariance function \eqref{eq:cov} with $r=1$, we have $\mu_n \lesssim n^{-2}$ 
and hence there is an intrinsic ordering of importance in the KL-expansion. Truncating the KL-expansion after $R$ terms, gives an approximation of $k$ in terms of $R$ standard normal random variables, 
\begin{equation}\label{truncated_kl}
k_R(\omega,x)=\exp\left[ \sum_{n=1}^R \sqrt{\mu_n}\phi_n(x)\xi_n(\omega)\right].
\end{equation}

Denote by $\vartheta := \{\xi_n\}_{n \in \mathbb N} \in \R^\mathbb N$
the vector of independent random variables appearing in the
KL-expansion of $\log k$. We will work with prior and posterior
measures on the space $\R^\mathbb N$. To this end, we equip
$\R^\mathbb N$ with the product sigma algebra $\mathcal B :=
\bigotimes_{n \in \mathbb N} \mathcal B^1(\R)$, where $\mathcal
B^1(\R)$ denotes the sigma algebra of Borel sets of $\R$. We denote by
$\rho_0$ the prior measure on $\R^\mathbb N$, defined by $\{\xi_n\}_{n
  \in \mathbb N}$ being independent and identically distributed
(i.i.d) $\mathcal N(0,1)$ random variables, such that
\begin{equation}\label{def:prior} 
\rho_0 = \bigotimes_{n \in \mathbb N} \, g(\xi_n) \, \mathrm{d} \xi_n,
\end{equation}
where $g : \R \rightarrow \R^+$ is the Lebesgue density of a $\mathcal
N(0,1)$ random variable and $d \xi_n$ denotes the one dimensional Lebesgue measure.

We assume that the observed data is finite
dimensional, i.e.~$F_\mathrm{obs} \in \R^m$ for some $m \in \mathbb
N$, and that
\begin{equation}\label{eq:data}
F_\mathrm{obs} = \mathcal F (p(\vartheta)) + \eta,
\end{equation}
where $\mathcal F: H^1(D) \rightarrow \R^m$ is a continuous function
of $p$, the (weak) solution to model problem \eqref{eq:mod} which
depends on $\vartheta$ through $k$. The observational noise $\eta$ is
assumed to be a realisation of a $\mathcal N(0, \sigma_F^2 I_m)$ random variable (independent of $\vartheta$). The parameter $\sigma_F^2$ is a fidelity parameter that indicates the level of observational noise present in $F_\mathrm{obs}$.

With $\rho_0$ as in \eqref{def:prior}, we have $\rho_0(\R^\N) =
1$. Furthermore, since $p$ depends continuously on $\vartheta$
  (see \cite[Propositions 3.6 and 4.1]{charrier12} or \cite[Lemmas
  2.20 and 5.13]{teckentrup_thesis}), the map $\mathcal F \circ p :
  \R^\N \rightarrow \R^m$ is also continuous (by assumption). The posterior distribution, which we will denote by $\rho$, is then known to be absolutely continuous with respect to the prior and satisfies 
\begin{equation}\label{eq:radnik}
\frac{\partial \rho}{\partial  \rho_0}(\vartheta) \, \eqsim \, \exp\left[-\frac{\|F_\mathrm{obs} - \mathcal F (p(\vartheta))\|^2}{2\sigma_F^2}\right] =: \exp\left[ - \Phi(\vartheta; F_\mathrm{obs})\right],
\end{equation}
where $\|\cdot\|$ denotes the Euclidean norm on $\R^m$. The hidden
constant depends only on $F_\mathrm{obs}$ and is generally not known
(for more details see \cite{stuart10} and the references therein). The right hand side of \eqref{eq:radnik} is referred to as the {\em likelihood}.

Since the exact solution $p(\vartheta)$ is not available, the
likelihood $\exp\left[ - \Phi(\vartheta; F_\mathrm{obs})\right]$ needs
to be approximated in practical computations. {We use a truncation 
of the KL-expansion of $\log k$ after $R$ terms and a spatial
approximation $p_{M,R}$  of $p(\vartheta)$ by piecewise linear FEs.
The value of $\sigma_F^2$ may also be changed to $\sigma_{F,M}^2$.} We denote the resulting approximate posterior measure correspondingly by $\rho^{M,R}$, with
\begin{equation}\label{eq:radnik_approx}
\frac{\partial \rho^{M,R}}{\partial \rho_0}(\vartheta) \, \eqsim \, \exp\left[-\frac{\|F_\mathrm{obs} - \mathcal F (p_{M,R}(\vartheta))\|^2}{2\sigma_{F,M}^2}\right] =: \exp\left[ - \Phi^{M,R}(\vartheta; F_\mathrm{obs})\right].
\end{equation}
Since {$\mathcal F  \circ p_{R,M}$} 
only depends on $\theta :=
\{\xi_n\}_{n=1}^R$, the first $R$ components of $\vartheta$, and since
the prior measure factorises as $\rho_0 = \rho_0^R \otimes
\rho_0^\perp$, the approximate posterior measure also factorises
 as $\rho^{M,R} = \nu^{M,R} \otimes \rho^\perp$, where
\begin{equation}\label{eq:radnik_approx_trunc}
\frac{\partial  \nu^{M,R}}{\partial \rho_0^R}(\theta) \, \eqsim \, \exp\left[ - \Phi^{M,R}(\theta; F_\mathrm{obs})\right],
\end{equation}
and $\rho^\perp = \rho_0^\perp$ \cite{ds11}. Note that $\nu^{M,R}$ is a measure on the finite dimensional space $\R^R$. Denoting by $\pi^{M,R}$ and $\pi_0^R$ the densities with respect to the $R$ dimensional Lebesgue measure of $\nu^{M,R}$ and $\rho_0^R$, respectively, it follows from \eqref{eq:radnik_approx_trunc} that
\begin{equation}\label{eq:bayes}
\pi^{M,R}(\theta) \eqsim \exp\left[ - \Phi^{M,R}(\theta; F_\mathrm{obs})\right] \, \pi_0^R (\theta)\, .
\end{equation}

Our goal is to approximate the expected value of a quantity $Q
= \mathcal G (p(\vartheta))$ with respect to the posterior $\rho$,
for some continuous $\mathcal G : H^1(D) \rightarrow \R$. We denote
this expected value by $\EE_\rho[Q] := \int_{\R^\N} \mathcal
G(p(\vartheta)) \, \rho(\mathrm{d} \vartheta)$ and assume that, as $M, R \rightarrow \infty$,
\[
\EE_{\nu^{M,R}}[Q_{M,R}] \rightarrow \EE_\rho[Q],
\]
where $\EE_{\nu^{M,R}}[Q_{M,R}] := \int_{\R^R} \mathcal
G(p_{M,R}(\theta)) \, \nu^{M,R}(\mathrm{d} \theta)$ is a finite
dimensional integral. 

Finally, let us set the notation for our MLMCMC algorithm. To achieve a level-dependent representation of~$k$, we simply truncate the KL-expansion after a sufficiently large, level-dependent number of terms~$R_\ell$, such that the truncation error on each level is bounded by the discretisation error,   
and set $\theta_\ell := \{\xi_n\}_{n=1}^{R_\ell}$. A sequence of
discretisation levels $M_\ell$ satisfying \eqref{eq:growthm} can be
constructed by choosing a coarsest mesh width $h_0$ for the spatial
approximation, and choosing $h_\ell := s^{-\ell} h_0$. A common (but
not necessarily optimal) choice is $s=2$ and uniform refinement
between the levels. We denote the resulting (truncated) {FE} solution by $p_\ell := p_{M_\ell,R_\ell}$.

The prior density $\pi_0^\ell$ of $\theta_\ell$ is simply a standard $R_\ell$-dimensional Gaussian:
\begin{equation}\label{eq:prior}
\pi_0^\ell(\theta_\ell) = \frac{1}{(2\pi)^{R_\ell/2}} \, \exp\left[ -\sum_{j=1}^{R_\ell}\frac{\xi_j^2}{2}\right].
\end{equation}
For the likelihood, we have
\begin{equation}\label{eq:like}
\mathcal L_\ell(F_\mathrm{obs}\, |\, \theta_\ell) \eqsim \exp\left[\frac{- \|F_\mathrm{obs} - F^\ell(\theta_\ell)\|^2}{2\sigma_{F,\ell}^2}\right],
\end{equation}
where $F^\ell(\theta_\ell) = \mathcal F(p_\ell(\theta_\ell))$.
Recall that the coarser levels in our multilevel estimator are
introduced only to accelerate the convergence and that the multilevel
estimator is still an unbiased estimator of the expected value of
$Q_L$ with respect to the posterior $\nu^L$ on the finest level
$L$. Hence, the posterior distributions on the coarser levels
$\nu^\ell$, $\ell=0,\ldots,L-1$, do not have to model the measured
data as faithfully as $\nu^L$. In particular, this means that we can
choose larger values of the fidelity parameter $\sigma_{F,\ell}^2$ on
the coarse levels, which will increase the acceptance probability on
the coarser levels. The growth in
$\sigma_{F,\ell}^2$ has to be controlled, 
as we will see below (cf. Assumption A3).

\subsection{Convergence analysis}
\label{sec:mod_conv}
We now perform a rigorous convergence analysis of the MLMCMC
  estimator $\widehat Q^\mathrm{ML}_{L,\{N_\ell\}}$ introduced in
  Section \ref{sec:mlmcmc} applied to model problem
  \eqref{eq:mod}. 
We will first
  verify that the multilevel estimator is indeed an unbiased
  estimator of $\EE_{\nu^L}[Q_L]$.
To achieve this, we only need to verify the irreducibility condition
\eqref{eq:mlirreducible} in Theorem~\ref{thm:mlmcconv}. As already
noted, for common choices of proposal
distributions, the condition holds true if
$\pi^{\ell-1}(\theta_{\ell,C}) > 0$, for all $\theta_\ell$
s.t.~$\pi^\ell(\theta_\ell) > 0$.  The conclusion follows, since both the prior and the likelihood were chosen as normal distributions and normal distributions have infinite support.
\begin{theorem}
Suppose that for all $\ell=0, \dots, L$, $\EE_{\nu^\ell}\left[ |Q_\ell| \right] < \infty$. Then
\begin{equation*}
\lim_{\{N_\ell\} \rightarrow \infty}  \widehat Q^\mathrm{ML}_{L, \{N_\ell\}} = \EE_{\nu^L}\left[ Q_L \right], \qquad \text{for any } \ \theta_\ell^0 \in \mathcal E^\ell \ \ \text{and} \ \ n_0^\ell \geq 0.
\end{equation*} 
\end{theorem}

Let us now move on to quantifying the cost of the multilevel estimator, and verify the assumptions in Theorem \ref{thm:main} for our model problem. As mentioned earlier, assumption M3 involves bounding the mean square error of an MCMC estimator, and a proof of M3 is beyond the scope of this paper. Results of this kind can be found in e.g.~\cite{rudolf_thesis, hsv11}. We will also not address M4, which is an assumption on the cost of obtaining one sample of $Q_\ell$. In the best case, with an optimal linear solver to solve the discretised (FE) equations for each sample, M4 is satisfied with $\gamma = 1$.

We will address assumptions M1 and M2, which are the assumptions related to the discretisation errors in the quantity of interest $Q$ and the measure $\rho$. For ease of presentation, we will for the remainder of this section assume that
$\log k$ has mean zero and exponential covariance function
\eqref{eq:cov} with $r=1$, and that $\psi$ and $f$ in \eqref{eq:mod2}
are deterministic, with $\psi \in H^1(\partial D)$ and $f \in
H^{-1/2}(D)$. This implies that the solution $p$ to \eqref{eq:mod2} is
in $L^q(\Omega, H^{3/2-\delta})$, for any $\delta > 0$ and $q <
\infty$ (cf. \cite{tsgu13}). In the Metropolis-Hastings
algorithm we will only consider symmetric proposal distributions or
the pCN algorithm.

Since they will become useful later, let us recall some of the main
results in the convergence analysis of (``plain vanilla'')
multilevel Monte Carlo estimators based on independent and identically
distributed (i.i.d.) samples.
An extensive convergence analysis of FE multilevel estimators based on i.i.d. samples 
for model problem \eqref{eq:mod2} with log--normal coefficients can be found in \cite{cst11,tsgu13,teckentrup12}.
We firstly have the following result on the convergence of the FE error in the natural $H^1$--norm.

\begin{theorem}\label{thm:fe_h1} Let $g$ be a Gaussian field with constant mean and covariance function \eqref{eq:cov} with $r=1$, and let $k=\exp[g]$ in model problem \eqref{eq:mod2}. Suppose $D \subset \mathbb R^d$ is Lipschitz polygonal (polyhedral). Then 
\begin{equation*}
\EE_{{\rho_0}}\left[|p - p_\ell|_{H^1(D)}^q \right]^{1/q} \leq C_{k,f,\psi,q} \, (M_\ell^{-1/2d + \delta} + R_\ell^{-1/2+\delta}),
\end{equation*}
for any $q < \infty$ and $\delta > 0$, where the (generic) constant $C_{k,f,\psi,q}$ (here and below) depends on the data $k$, $f$, $\psi$ and on $q$, but is independent of any other parameters.
\end{theorem}
\begin{proof}
This follows from \cite[Proposition 4.1]{tsgu13}.
\end{proof}

Convergence results for functionals of the solution $p$ can now be
derived from Theorem~\ref{thm:fe_h1} using a duality argument. We will
here for simplicity only consider bounded, linear functionals, but the
results {extend to continuously Fr\`echet differentiable
  functionals} (see \cite[\S 3.2]{tsgu13}). We make the following assumption on the functional $\mathcal G$ (cf. Assumption F1 in \cite{tsgu13}).

\begin{itemize}
\item[{\bf A2.}] Let $\mathcal G : H^1(D) \rightarrow \mathbb R$ be linear, and suppose there exists $C_{\mathcal G} \in \R$, such that
\[
|\mathcal G(v)| {\;\le\;} C_{\mathcal G} \|v\|_{H^{1/2-\delta}}, \qquad \text{for all $\delta > 0$.}
\]
\end{itemize}
An example of a functional which satisfies A2 is a local average of
the pressure, {$\frac{1}{|D^*|} \int_{D^*} p \dx$ for some $D^* \subset D$.}
The main result on the convergence for functionals is the following.

\begin{corollary}\label{lem:fe_func} Let the assumptions of Theorem \ref{thm:fe_h1} be satisfied, and suppose $\mathcal G$ satisfies A2. Then  
\[
\EE_{{\rho_0}}\left[|\mathcal G(p) - \mathcal G(p_\ell)|^q \right]^{1/q} \ \leq  C_{k,f,\psi,q} \, \left(M_\ell^{-1/d + \delta} + R_\ell^{-1/2+\delta}\right),    
\]
for any $q < \infty$ and $\delta > 0$.
\end{corollary}
\begin{proof} This follows from \cite[Corollary 4.1]{tsgu13}.
\end{proof}

Note that assumption A2 is crucial in order to get the faster convergence rates of the spatial discretisation error in Corollary \ref{lem:fe_func}. For multilevel estimators based on i.i.d. samples, it follows immediately from Corollary \ref{lem:fe_func} that the (corresponding) assumptions M1 and M2 are satisfied, with $\alpha=1/d+\delta$, $\alpha'=1/2+\delta$ and $\beta=2\alpha$, $\beta'=2\alpha'$, for any $\delta>0$ (see \cite{tsgu13} for details).

The aim is now to generalise the result in Corollary \ref{lem:fe_func} to
the new MLMCMC estimator. Two issues need to be addressed.
Firstly, the bounds in assumptions M1 and M2 in Theorem \ref{thm:main}
involve moments with respect to the posterior distributions $\nu^\ell$
and $\rho$, which are not known explicitly, but are related to the
prior distributions $\rho_0^\ell$ and $\rho_0$ through Bayes'
Theorem. Secondly, the samples on levels {$\ell$ and $\ell-1$ that
  are used to compute samples of the differences $Y_\ell = Q_\ell -
  Q_{\ell-1}$ are generated by Algorithm~2, and may differ not only
  due to discretisation and truncation order,} but also because they come from different Markov chains (i.e. $\Theta_{\ell-1}^n$ is not necessarily equal to $\theta_{\ell,C}^n$, as seen in Table \ref{tab:outcomes}).

To circumvent the problem of the intractability of the posterior distribution, we have the following lemma, which relates moments with respect to the posterior distribution to moments with respect to the prior distribution.

\begin{lemma}\label{lem:post_bound} For any random variable $Z = Z(\theta_\ell)$ and for any $q$ s.t. $\EE_{\rho_0^\ell}\left[ |Z|^q \right] < \infty$, we have
\begin{align*}
\left| \, \EE_{\nu^\ell}\left[ Z^q \right] \, \right| \, &\lesssim \, \EE_{\rho^\ell_0}\left[ |Z|^q \right].
\end{align*}
Similarly, for any random variable $Z = Z(\vartheta)$ {and for any $q$ s.t. $\EE_{\rho_0}\left[ |Z|^q \right] < \infty$}, we have
\begin{align*}
\left| \, \EE_{\rho^\ell}\left[ Z^q \right] \, \right| \, &\lesssim \, \EE_{\rho_0}\left[ |Z|^q \right].
\end{align*}
\end{lemma}
\begin{proof} Using \eqref{eq:bayes}, we have 
\begin{align*}
\left| \, \EE_{\nu^\ell}\left[ Z^q \right] \, \right| 
&\;\eqsim \; \left|\int_{\mathbb R^{R_\ell}} Z^q(\theta_\ell) \exp\left[ - \Phi^{M,R}(\theta_\ell; F_\mathrm{obs})\right] \, \pi_0^\ell(\theta_\ell) \, \dtheta_\ell \right| \\ \vspace*{2ex}
&\;\lesssim\;  \sup_{\theta_\ell}  \left\{\exp\left[ - \Phi^{M,R}(\theta_\ell; F_\mathrm{obs})\right] \right\}\, \int_{\mathbb R^{R_\ell}} |Z(\theta_\ell)|^q \, \pi_0^\ell(\theta_\ell) \dtheta_\ell. 
\end{align*}
The first claim of the Lemma then follows, since the above supremum can be bounded by 1. The proof of the second claim is analogous, using the Radon-Nikodym derivative \eqref{eq:radnik}.
\end{proof}

We are now ready to prove assumption M1, under the following assumption on the parameters $\sigma_{F,\ell}^2$ in the likelihood model \eqref{eq:like}:
\begin{itemize}
\item[{\bf A3.}] The sequence of fidelity parameters 
$\{\sigma_{F,\ell}^2\}_{\ell=0}^\infty$ satisfies
\[
\sigma_{F}^{-2} - \sigma_{F,\ell}^{-2} \;  \lesssim \; 
\max\left(R_{\ell}^{-1/2+\delta},M_{\ell}^{-1/d+\delta}\right), \quad \text{for all} \ \ \delta>0.
\]
\end{itemize}

\begin{lemma} \label{thm:m1} Let the assumptions of Corollary \ref{lem:fe_func} be satisfied. Suppose $\mathcal F$ satisfies A2, and A3 holds. Then 
\[
|\EE_{\nu^\ell}[Q_\ell] - \EE_{\rho}[Q]| \leq C_{k,f,\psi} \, \left(M_\ell^{-1/d + \delta} + R_\ell^{-1/2+\delta}\right).
\]
\end{lemma} 
\begin{proof}
Since $Q_\ell$ only depends on $\theta_\ell$ we have
$\EE_{\nu^\ell}[Q_\ell] = \EE_{\rho^\ell}[Q_{\ell}]$ and so, using the triangle inequality,
\begin{equation}
\label{eq:bias_triangle}
|\EE_{\nu^\ell}[Q_\ell] - \EE_{\rho}[Q]| \leq |\EE_{\rho^\ell}[Q_\ell]
- \EE_{\rho^\ell}[Q]| + |\EE_{\rho^\ell}[Q] - \EE_{\rho}[Q]|.
\end{equation}
The first term can be bounded using Corollary \ref{lem:fe_func} and
Lemma \ref{lem:post_bound}, i.e.
\[
|\EE_{\rho^\ell}[Q_\ell] - \EE_{\rho^\ell}[Q]| \leq C_{k,f,\psi} \, \left(M_\ell^{-1/d + \delta} + R_\ell^{-1/2+\delta}\right).
\]

For the second term, we will prove a bound on the Hellinger distance $d_\mathrm{Hell}(\rho,\rho^\ell)$. This proof follows closely the proof of \cite[Proposition 10]{hss12}. Denote by $Z$ and $Z_\ell$ the normalising constants of $\rho$ and $\rho^\ell$:
\[
Z =  \int_{\R^\N} \exp\left[ - \frac{1}{2}\Phi(\vartheta;
  F_\mathrm{obs})\right] \mathrm{d}\rho_0(\vartheta) \quad \text{and}
\quad Z_\ell =  \int_{\R^\N} \exp\left[ -
  \frac{1}{2}\Phi^{\ell}(\vartheta; F_\mathrm{obs})\right]
\mathrm{d}\rho_0(\vartheta), \quad \text{respectively}.
\]
Since $\mathcal F$ satisfies Assumption A2, it follows from the results in \cite{stuart10} that both $Z$ and $Z_\ell$ can be bounded away from zero. Next, we have
\[
2\,d_{\mathrm{Hell}}^2(\rho,\rho^\ell) = \int_{\R^\N} \left( Z^{-1/2}  \exp\left[ - \frac{1}{2}\Phi(\vartheta; F_\mathrm{obs})\right] - Z_\ell^{-1/2} \exp\left[ - \frac{1}{2}\Phi^{\ell}(\vartheta; F_\mathrm{obs})\right] \right)^2 \mathrm{d}\rho_0(\vartheta) \leq I + II,
\]
where\vspace{-1.5ex}
\begin{align*}
I &:= \frac{2}{Z} \int_{\R^\N} \left( \exp\left[ - \frac{1}{2}\Phi(\vartheta; F_\mathrm{obs})\right] - \exp\left[ - \frac{1}{2}\Phi^{\ell}(\vartheta; F_\mathrm{obs})\right] \right)^2 \mathrm{d}\rho_0(\vartheta), \\
II &:= 2 \, |Z^{-1/2} - Z_\ell^{-1/2}|^2 \int_{\R^\N} \exp\left[ -\Phi^{\ell}(\vartheta; F_\mathrm{obs})\right] \mathrm{d}\rho_0(\vartheta).
\end{align*}
To estimate I, note that both $\exp\left[ - \frac{1}{2}\Phi(\vartheta;
  F_\mathrm{obs})\right]$ and $\exp\left[ -
  \frac{1}{2}\Phi^{\ell}(\vartheta; F_\mathrm{obs})\right]$ are
bounded above by 1, so that
\[
\exp\left[ - \frac{1}{2}\Phi(\vartheta; F_\mathrm{obs})\right] - \exp\left[ - \frac{1}{2}\Phi^{\ell}(\vartheta; F_\mathrm{obs})\right] \leq |\Phi(\vartheta; F_\mathrm{obs}) - \Phi^{\ell}(\vartheta; F_\mathrm{obs})|.
\]
Denoting $F := \mathcal F(p(\vartheta))$ and $F_{\ell} := \mathcal F (p_\ell(\theta))$, and using the triangle inequality, we have that 
\begin{align*}
\left|\frac{\|F_\mathrm{obs} - F\|^2}{\sigma_{F}^2} - \frac{\|F_\mathrm{obs} - F_{\ell}\|^2}{\sigma_{F,\ell}^2}\right| &\ \le \ \left|{\frac{\Big(\|F_\mathrm{obs} - F_{\ell}\| + \|F - F_{\ell}\|\Big)^2}{\sigma_{F}^2} - \frac{\|F_\mathrm{obs} - F_{\ell}\|^2}{\sigma_{F,\ell}^2}}\right|
\\[1ex]
&\hspace{-3.5cm} = \ \|F_\mathrm{obs} - F_{\ell}\|^2 \left(\sigma_{F}^{-2} - \sigma_{F,\ell}^{-2} \right) \; + \; \frac{2 \|F_\mathrm{obs} - F_{\ell}\| + \|F - F_{\ell}\|}{\sigma_{F}^2}  \|F - F_{\ell}\|.
\end{align*}
Since $\mathcal F$ was assumed to satisfy A2, it follows from Corollary \ref{lem:fe_func} that
\begin{align*}
\EE_{\rho_0}[\|F - F_{\ell}\|^q]^{1/q} 
&\ \le \ C_{k,f,\psi} \left(M_\ell^{-1/d+\delta} \, + \, R_\ell^{-1/2 + \delta} \right).
\end{align*}
Moreover, since $\|F_{\ell}\|$ can be bounded independently of $\ell$ (again courtesy of Assumption~A2), and since 
$\|F_\mathrm{obs} - F_{\ell}\| \le \|F_\mathrm{obs}\| + \|F_{\ell}\|$, 
we can deduce that
\begin{align*}
I \ \lesssim \ \EE_{\rho_0}[ |\Phi(\vartheta; F_\mathrm{obs}) -
\Phi^{\ell}(\vartheta; F_\mathrm{obs})|^2]  \ \le \ C_{k,f,\psi} \left(M_\ell^{-1/d+\delta} \, + \, R_\ell^{-1/2 + \delta} \right)^2.
\end{align*}
using Assumption A3. For the second term II, we note that $ |Z^{-1/2} - Z_\ell^{-1/2}|^2 \lesssim \max\{Z^{-3},Z_\ell^{-3}\} \;  |Z - Z_\ell|^2$,
and an analysis similar to the above shows that
\[
II \ \lesssim \ \EE_{\rho_0}[ |\Phi(\vartheta; F_\mathrm{obs}) - \Phi^{\ell}(\vartheta; F_\mathrm{obs})|]^2 \ \le \ C_{k,f,\psi} \left(M_\ell^{-1/d+\delta} \, + \, R_\ell^{-1/2 + \delta} \right)^2.
\]
The claim of the Theorem then follows, since $
|\EE_{\rho^\ell}[Q] - \EE_{\rho}[Q]| \leq C_{k,f,\psi} d_\mathrm{Hell}(\rho,\rho^\ell)$.
\end{proof}

In order to prove M2, we further have to  analyse the situation where
the two samples $\theta_\ell^{n+1}$ and $\Theta_{\ell-1}^{n+1}$ used to
compute $Y_\ell^{n+1}$ ``diverge'', i.e. when $\Theta_{\ell-1}^{n+1} \neq
  \theta_{\ell,C}^{n+1}$. 

For the remainder we will consider only symmetric or pCN
proposal distributions.

\begin{lemma}\label{lem:varconv} 
Let $\theta_\ell^{n+1}$ and $\Theta_{\ell-1}^{n+1}$ have joint
distribution $\nu^{\ell, \ell-1}$, and set $Y_\ell^{n+1} = Q_\ell(\theta_\ell^{n+1}) - Q_{\ell-1}(\Theta_{\ell-1}^{n+1})$. 
If $q _\mathrm{ML}^{\ell,F}$ is a pCN proposal distribution, then
\begin{equation*}
\VV_{\nu^{\ell, \ell-1}}\left[Y_\ell^{n+1}\right] \leq  C_{k,f,\psi} \, {\left( M_{\ell-1}^{-1/d+\delta} + R_{\ell-1}^{-1/2+\delta}\right)}, \quad \text{for any} \ \ \delta > 0.
\end{equation*}
This bound also holds for a symmetric proposal distribution
$q _\mathrm{ML}^{\ell,F}$ under the additional assumption that
\begin{equation}
\label{assump:R}
(R_\ell - R_{\ell-1}) (2\pi)^{-\frac{R_\ell - R_{\ell-1}}{2}} \; \lesssim \; 
R_{\ell-1}^{-1/2+\delta}, \quad \text{for all} \ \ \delta>0.
\end{equation}
\end{lemma}

For the growth condition \eqref{assump:R} to
be satisfied, it suffices that $R_\ell - R_{\ell-1}$ grows
logarithmically with $R_{\ell-1}.$ To prove Lemma \ref{lem:varconv},
we first need some preliminary results. Firstly, note that
$\Theta_{\ell-1}^{n+1} \neq \theta_{\ell,C}^{n+1}$ only if the
proposal $\theta_\ell'$ generated for $\theta_\ell^{n+1}$ was rejected. 
Given the states $\theta_\ell^{n}$ and $\theta_\ell'$, the probability of
this rejection is given by $1-\alpha^\ell_\mathrm{ML}(\theta_\ell'|\theta_\ell^{n})$. The total probability of a rejection is then $\EE_{\boldsymbol{\zeta}}[(1 - \alpha^\ell_\mathrm{ML}]$, where $\zeta$ denotes the joint distribution of the two variables. We need to quantify this probability. 

Before
  we can do so, we need to specify the (marginal) distribution of the
  proposal $\theta_\ell'$, which we denote by $\zeta_\ell'$. The first $R_{\ell-1}$ entries of $\theta_\ell'$ are
  distributed as $\nu^{\ell-1}$, since they come from
  $\Theta_{\ell-1}$. The remaining $R_\ell - R_{\ell-1}$ dimensions
  are distributed according to the proposal density
  $q _\mathrm{ML}^{\ell,F}(\theta_{\ell,F}' \, | \, \theta_{\ell,F}^n)$
  (independent of the first $R_{\ell-1}$ dimensions). The same proof
  technique as in Lemma \ref{lem:post_bound} shows again that
  $|\EE_{{\zeta_\ell'}}[Z^q]| \lesssim \EE_{\rho_0^\ell}[|Z|^q]$,
  for any random variable $Z = Z(\theta_\ell')$.

\begin{lemma}\label{thm:alphaconv} Let $\theta_\ell^n$ and
  $\theta_\ell'$ be
  as generated by Algorithm 2 at the $(n+1)$th step. Denote their joint distribution by
  $\boldsymbol{\zeta}$, with marginal distributions $\nu^\ell$ and $\zeta_\ell'$, 
  respectively. Suppose $\mathcal F$ satisfies A2, and A3 and the assumptions of Corollary \ref{lem:fe_func} hold.
If $q _\mathrm{ML}^{\ell,F}$ is a pCN proposal distribution, then
\begin{equation*}
\EE_{\boldsymbol{\zeta}}\Big[(1 - \alpha^\ell_\mathrm{ML}(\theta_\ell'|\theta_\ell^{n}))\Big] \leq  C_{k,f,\psi}
{\left(M_{\ell-1}^{-1/d+\delta} + R_{\ell-1}^{-1/2+\delta}  \right)},
\quad \text{for any} \ \ \delta > 0.
\end{equation*}
This bound also holds for
a symmetric proposal distribution $q _\mathrm{ML}^{\ell,F}$ under the additional assumption \eqref{assump:R}.
\end{lemma}
\begin{proof} 
We will start by assuming that $q _\mathrm{ML}^{\ell,F}$ is a pCN proposal distribution. For brevity, denote $\mathcal L_\ell(F_\mathrm{obs}\, |\, \cdot) =: \mathcal L_\ell(\cdot)$.
We will first derive a bound on $1-\alpha^\ell_\mathrm{ML}(\theta_\ell' \, | \, \theta_\ell^n)$, for $\ell > 1$ and for $\theta_\ell'$ and $\theta_\ell^n$ given. First note that if $\frac{\mathcal L_\ell(\theta_\ell') \,  \mathcal L_{\ell-1}(\theta_{\ell,C}^n)}{\mathcal L_\ell(\theta_\ell^n) \,  \mathcal L_{\ell-1}(\theta_{\ell,C}')} \geq 1$, then $1-\alpha^\ell(\theta_\ell' \, | \, \theta_\ell^n) = 0$. Otherwise, we have 
\begin{align}
1-\alpha^{\ell}_\mathrm{ML}(\theta_{\ell}' \, | \, \theta_{\ell}^n)  
&= \left(1-\frac{\mathcal L_\ell(\theta_\ell')}{\mathcal L_{\ell-1}(\theta_{\ell,C}')} \right) + \left( \frac{\mathcal L_\ell(\theta_\ell') \,  \mathcal L_{\ell-1}(\theta_{\ell,C}^n)}{\mathcal L_\ell(\theta_\ell^n) \,  \mathcal L_{\ell-1}(\theta_{\ell,C}')} \right) \left(1-\frac{\mathcal L_\ell(\theta_\ell^n)}{\mathcal L_{\ell-1}(\theta_{\ell,C}^n)}\right) \nonumber \\
&\leq \left| 1-\frac{\mathcal L_\ell(\theta_\ell')}{\mathcal L_{\ell-1}(\theta_{\ell,C}')}\right| + \left| 1-\frac{\mathcal L_\ell(\theta_\ell^n)}{\mathcal L_{\ell-1}(\theta_{\ell,C}^n)}\right|.
\label{eq:alphaproof0}
\end{align}
Let us consider either of these two terms and set $\theta_\ell = (\xi_j)_{j=1}^{R_\ell}$ to be either $\theta_\ell'$ or $\theta_\ell^n$. Using the definition \eqref{eq:like} of the likelihood, we have
\begin{equation}
\label{eq:alphaproof}
\frac{\mathcal L_\ell(\theta_\ell)}{\mathcal L_{\ell-1}(\theta_{\ell,C})} \; = \; \exp \left(- \; \frac{\|F_\mathrm{obs} - F_\ell(\theta_\ell)\|^2}{\sigma_{F,\ell}^2} \;+\; \frac{\|F_\mathrm{obs} - F_{\ell-1}(\theta_{\ell,C})\|^2}{\sigma_{F,\ell-1}^2}\right). 
\end{equation}
Denoting $F_\ell := F(\theta_\ell)$ and $F_{\ell-1} := F(\theta_{\ell,C})$, 
we get as in the proof of Lemma \ref{thm:m1} that 
\begin{align}\label{eq:alphaproof2}
\left|\frac{\|F_\mathrm{obs} - F_\ell\|^2}{\sigma_{F,\ell}^2} -
  \frac{\|F_\mathrm{obs} - F_{\ell-1}\|^2}{\sigma_{F,\ell-1}^2}\right|
& \leq \ \|F_\mathrm{obs} - F_{\ell-1}\|^2 \left|\sigma_{F,\ell}^{-2}
  - \sigma_{F,\ell-1}^{-2} \right| \nonumber \\
&
+ \ \frac{2 \|F_\mathrm{obs} - F_{\ell-1}\| + \|F_\ell - F_{\ell-1}\|}{\sigma_{F,\ell}^2}  \|F_\ell - F_{\ell-1}\|.
\end{align}
Using the inequality $|1-\exp(x)| \leq |x|$, for $0 \leq |x| \leq 1$, it follows immediately from \eqref{eq:alphaproof2}, Assumption A3, Corollary \ref{lem:fe_func}, Lemma \ref{lem:post_bound} and H\"olders inequality that  
\begin{align}
\label{eq:alphaproof1}
\EE_{\boldsymbol{\zeta}} \left[\left| 1-\frac{\mathcal L_\ell(\theta_\ell)}{\mathcal L_{\ell-1}(\theta_{\ell,C})}\right|\right]
\; \le \; C_{k,f,\psi} \left( M_{\ell-1}^{-1/d+\delta} \, + \, R_{\ell-1}^{-1/2+\delta}\right).\vspace{-3ex}
\end{align}
A bound on the expected value of $1-\alpha^{\ell}_\mathrm{ML}(\theta_{\ell}' \, | \, \theta_{\ell}^n)$ now follows from Minkowski's inequality.

The proof in the case of a symmetric proposal distribution is analogous. The bound \eqref{eq:alphaproof0} is replaced by
\[
1-\alpha^{\ell}_\mathrm{ML}(\theta_{\ell}' \, | \, \theta_{\ell}^n)
\leq \left| 1-\frac{\pi^\ell(\theta_\ell')}{\pi^{\ell-1}(\theta_{\ell,C}')}\right| + \left| 1-\frac{\pi^\ell(\theta_\ell^n)}{\pi^{\ell-1}(\theta_{\ell,C}^n )}\right|.
\]
Using the definition of $\pi^\ell$ in \eqref{eq:bayes}, as well as the
models \eqref{eq:prior} and \eqref{eq:like} for the prior and the
likelihood, respectively, we have instead of \eqref{eq:alphaproof} that
\begin{align} \label{eq:alphaproof_sym}
\frac{\pi^\ell(\theta_\ell)}{\pi^{\ell-1}(\theta_{\ell,C})} & \;=\; 
\frac{\pi_0^\ell(\theta_\ell) \qquad \mathcal L_\ell(\theta_\ell)}{\pi_0^{\ell-1}(\theta_{\ell,C}) \, \mathcal L_{\ell-1}(\theta_{\ell,C})} \\
&\; = \; \exp \left(-\,(2\pi)^{-\frac{R_\ell-R_{\ell-1}}{2}} \sum_{j=R_{\ell-1}+1}^{R_\ell} \frac{\xi_j^2}{2} \;- \; \frac{\|F_\mathrm{obs} - F_\ell(\theta_\ell)\|^2}{\sigma_{F,\ell}^2} \;+\; \frac{\|F_\mathrm{obs} - F_{\ell-1}(\theta_{\ell,C})\|^2}{\sigma_{F,\ell-1}^2}\right). \nonumber
\end{align}
Since $\sum_{j=R_{\ell-1}+1}^{R_\ell} \xi_j^2$ is $\chi^2$-distributed with $R_\ell-R_{\ell-1}$ degrees of freedom, we have
\[
\EE_{\rho_0^\ell} \Big[ {\textstyle \sum_{j} \xi_j^2} \Big] = 2 \frac{\Gamma\left(\frac{1}{2}(R_\ell-R_{\ell-1}) + 1\right)}{\Gamma\left(\frac{1}{2}(R_\ell-R_{\ell-1})\right)} \lesssim R_\ell-R_{\ell-1}.
\]
Together with the assumption in \eqref{assump:R} this implies that the
expected value of the additional term in \eqref{eq:alphaproof_sym} is
bounded by $R_{\ell-1}^{-1/2+\delta}$. The proof then reduces to that
in the pCN case above.
\end{proof}

We will further need the following result.

\begin{lemma}\label{lem:thetadiff}
For any $\theta_\ell$, let 
$k_\ell(\theta_\ell) := \exp \left(\sum_{j=1}^{R_\ell} \sqrt{\mu_j}
  \phi_j (\theta_\ell)_j\right)$ and $\kappa(\theta_\ell) := \min_{x \in
  \overline D} k_\ell(\cdot,x)$. 
Then
\begin{equation}
\label{eq:thetadiff1}
|p_\ell(\theta_\ell) - p_\ell(\theta_\ell^*)|_{H^1(D)} \;\lesssim\; 
\frac{\|f\|_{H^{-1}(D)}}{\kappa(\theta_\ell)\kappa(\theta_\ell^*)}\,
\|k_\ell(\theta_\ell) - k_\ell(\theta_\ell^*)\|_{\mathcal C^0(\overline D)}, \quad 
\text{for almost all} \ \ \theta_\ell, \theta_\ell^*,
\end{equation}
and\vspace{-1.5ex}
\begin{equation}
\label{eq:thetadiff2}
\EE_{\rho^\ell_0}\left[|p_\ell(\theta_\ell)|_{H^1(D)}^q\right] \ \le \ \mathrm{constant},
\end{equation}
for any $q < \infty$, where the hidden constants are independent of 
$\ell$ and $p_\ell$.
\end{lemma}
\begin{proof}
Using the definition of $\kappa(\theta_\ell)$, as well as the identity
$$
\int_D k_\ell(\theta_\ell) \nabla p_\ell(\theta_\ell) \cdot \nabla v \dx = \int_D fv \dx = \int_D k_\ell(\theta_\ell^*) \, \nabla p_\ell(\theta_\ell^*) \cdot \nabla v \dx, \qquad \text{for all} \ v \in H^1_0(D),
$$ 
we have\vspace{-1.5ex} 
\begin{align*}
\quad \kappa(\theta_\ell) |p_\ell(\theta_\ell) - p_\ell(\theta_\ell^*)|^2_{H^1(D)} 
&\leq\;  \int_D k_\ell(\theta_\ell) \, \nabla \left(p_\ell(\theta_\ell) - p_\ell(\theta_\ell^*)\right) \cdot \nabla \left(p_\ell(\theta_\ell) - p_\ell(\theta_\ell^*)\right) \dx\\
&\leq \int_D \left(k_\ell(\theta_\ell) - k_\ell(\theta_\ell^*)\right) \, \nabla p_\ell(\theta_\ell^*) \cdot \nabla \left(p_\ell(\theta_\ell) - p_\ell(\theta_\ell^*)\right) \dx. 
\end{align*}
Due to the standard estimate $|p_\ell(\theta_\ell^*)|_{H^1(D)}
\le \|f\|_{H^{-1}(D)}/\kappa(\theta_\ell^*)$, \eqref{eq:thetadiff1} follows from an application  of the Cauchy-Schwarz inequality, and \eqref{eq:thetadiff2} follows from the fact that $\EE_{\rho^\ell_0}\left[\kappa(\cdot)^{-q}\right]$ is bounded independent
of~$\ell$ (\cite[Prop.~3.10]{charrier12}).
\end{proof}

Using Lemmas \ref{thm:alphaconv} and \ref{lem:thetadiff}, we
are now ready to prove Lemma \ref{lem:varconv}.

\begin{proof}[Proof of Lemma \ref{lem:varconv}]
Let $\theta_\ell^{n+1}$ and $\Theta_{\ell-1}^{n+1}$ be as generated by Algorithm 2 at the $(n+1)$th step, with joint distribution $\nu^{\ell,\ell-1}$. As before, denote the proposal generated for $\theta_{\ell}^n$ by $\theta_\ell'$.
Firstly, since $\theta_{\ell,C}' = \Theta_{\ell-1}^{n+1}$, it follows from Minkowski's inequality that
\begin{align} 
\VV_{\nu^{\ell,\ell-1}}\left[Y_\ell^{n+1} \right] &\ \le \ 
\EE_{\nu^{\ell,\ell-1}}\left[\left(Q_\ell(\theta_\ell^{n+1}) -
    Q_{\ell-1}(\Theta^{n+1}_{\ell-1})\right)^2\right] \nonumber \\
\label{eq:mainlem_triangle}
& \ \lesssim \ \EE_{\widetilde{\boldsymbol{\zeta}}}\left[\left(Q_\ell(\theta_\ell^{n+1}) - Q_{\ell}(\theta_\ell')\right)^2\right] \,+\, \EE_{\zeta_\ell'}\left[\left(Q_\ell(\theta_\ell') - Q_{\ell-1}(\theta_{\ell,C}')\right)^2\right].
\end{align}
Here,
$\widetilde{\boldsymbol{\zeta}}$ denotes the
joint distribution of $\theta_\ell'$ and $\theta_\ell^{n+1}$ and
$\zeta_\ell'$ is the marginal distribution of $\theta_\ell'$.
A bound on the second term follows immediately from Corollary
  \ref{lem:fe_func} and Lemma \ref{lem:post_bound}, i.e.
\begin{equation}\label{eq:b1}
\EE_{\zeta_\ell'}\left[\left(Q_\ell(\theta_\ell') -
    Q_{\ell-1}(\theta_{\ell,C}')\right)^2\right] 
 \,\lesssim \,  \EE_{\rho_0^\ell}\left[\left(Q_\ell(\theta_\ell') - Q_{\ell-1}(\theta_{\ell,C}')\right)^2\right]
\,\leq\, C_{k,f,\psi} \, \left(M_{\ell-1}^{-1/d + \delta} + R_{\ell-1}^{-1+\delta}\right).
\end{equation}
The first term in \eqref{eq:mainlem_triangle} is nonzero only if $\theta_\ell^{n+1} \not= \theta_\ell'$. We will now use Lemmas \ref{thm:alphaconv} and \ref{lem:thetadiff}, as well as the characteristic function $\mathbb I_{\{\theta_\ell^{{n+1}} \neq \theta_\ell' \}} \in \{0,1\}$ to bound it. Firstly,  H\"older's inequality gives
\begin{align}\label{eq:b3}
\EE_{\widetilde{\boldsymbol{\zeta}}}\left[\left(Q_\ell(\theta_\ell^{n+1}) - Q_{\ell}(\theta_\ell')\right)^2\right] &\;= \;\EE_{\widetilde{\boldsymbol{\zeta}}}\left[\left(Q_\ell(\theta_\ell^{n+1}) - Q_{\ell}(\theta_\ell')\right)^2 \, \mathbb I_{\{\theta_\ell^{n+1} \neq \theta_\ell'\}}\right] \nonumber \\
&\;\leq\; \EE_{\widetilde{\boldsymbol{\zeta}}}\left[\left(Q_\ell(\theta_\ell^{n+1}) - Q_{\ell}(\theta_\ell')\right)^{2q_1}\right]^{1/q_1} \, \EE_{\widetilde{\boldsymbol{\zeta}}}\left[\mathbb I_{\{\theta_\ell^{n+1} \neq \theta_\ell'\}}\right]^{1/q_2},
\end{align}
for any $q_1, q_2$ s.t. $q_1^{-1} + q_2^{-1}=1$. Since $\mathcal G$ satisfies assumption A2, it follows from Lemmas \ref{lem:post_bound} and \ref{lem:thetadiff} that the term $\EE_{{\widetilde{\boldsymbol{\zeta}}}}\big[\left(Q_\ell(\theta_\ell^{n+1}) - Q_{\ell}(\theta_\ell')\right)^{2q_1}\big]^{1/q_1}$ in \eqref{eq:b3} can be bounded by a constant independent of $\ell$, for any $q_1 < \infty$:
\[
\EE_{{\widetilde{\boldsymbol{\zeta}}}}\big[\left(Q_\ell(\theta_\ell^n) - Q_{\ell}(\theta_\ell')\right)^{2q_1}\big] \lesssim \EE_{\nu^\ell}\big[(Q_\ell(\theta_\ell^{n+1}))^{2q_1}\big] + \EE_{\zeta_\ell'}\big[(Q_{\ell}(\theta_\ell'))^{2q_1}\big] \lesssim \EE_{\rho_0^\ell}\big[|p_\ell(\theta_\ell)|_{H^1(D)}^{2q_1}\big] \leq \text{constant}.
\] 
Since $\theta_\ell^{n+1} \neq \theta_\ell'$ only if the proposal
$\theta_\ell'$ has been  rejected on level $\ell$ at the $(n+1)$th
step, the probability that this happens can be bounded by
$\EE_{{\boldsymbol{\zeta}}}[1-
  \alpha_\text{ML}^\ell(\theta_\ell'|\theta_\ell^{n})]$, where the joint distribution
  $\boldsymbol{\zeta}$ is as in Lemma \ref{thm:alphaconv}. It follows by Lemma \ref{thm:alphaconv} that
\begin{equation}\label{eq:b4}
\EE_{{\widetilde{\boldsymbol{\zeta}}}}\left[\mathbb
  I_{\{\theta_\ell^{{n+1}} \neq \theta_\ell'\}}\right] \;=\;
\PP[\theta_\ell^{{n+1}} \neq \theta_\ell'] 
\; \leq\; C_{k,f,\psi} \, \left(M_{\ell-1}^{-1/d + \delta} + R_{\ell-1}^{-1+\delta}\right).
\end{equation}
Combining \eqref{eq:mainlem_triangle}-\eqref{eq:b4} 
the claim of the Lemma then follows.
\end{proof}

We now collect the results in the preceding lemmas to state our main result of this section.

\begin{theorem}\label{thm:rates} 
Under the same assumptions as in Lemma \ref{lem:varconv}, Assumptions M1 and M2 of Theorem~\ref{thm:main} are satisfied, with $\alpha=\beta=1/d-\delta$ and $\alpha'=\beta'=1/2-\delta$, for any $\delta>0$.
\end{theorem}

If we assume that we can obtain individual samples in optimal cost 
$\mathcal{C}_\ell \lesssim M_\ell \log (M_{\ell})$, e.g. via a
multigrid solver, we can satisfy Assumption M4 with $\gamma=1+\delta$,
for any $\delta > 0$. It follows from Theorems \ref{thm:main} and 
\ref{thm:rates} that we can get the following theoretical upper bounds
for the $\eps$-costs of classical and multilevel MCMC applied to model
problem \eqref{eq:mod2} with log-normal coefficients $k$:
\begin{equation}
\mathcal{C}_\eps(\widehat{Q}^\mathrm{MC}_{N}) \;\lesssim\; \varepsilon^{-(d+2)-\delta} \quad \text{and} \quad \mathcal{C}_\eps(\widehat{Q}^\mathrm{ML}_{L, \{N_\ell\}})  \;\lesssim\; \varepsilon^{-(d+1)-\delta}, \quad \text{for any} \ \delta > 0.
\end{equation}

We clearly see the advantages of the multilevel method, which gives a saving of one power of~$\eps^{-1}$ compared to the standard MCMC method. Note that for multilevel estimators based on i.i.d samples, the savings of the multilevel method over the standard method are two powers of $\eps^{-1}$, for $d=2,3$. The larger savings stem from the fact that $\beta=2\alpha$ in this case, compared to $\beta=\alpha$ in the MCMC analysis above. 
The numerical results in the next section for $d=2$ show that in
practice we do seem to observe $\beta \approx 1 \approx 2\alpha$,
leading to $\mathcal{C}_\eps(\widehat{Q}^\mathrm{ML}_{L, \{N_\ell\}})
= \mathcal{O}(\varepsilon^{-2})$. However, we do not believe that this
is a lack of sharpness in our theory, but rather a pre-asymptotic
phase. The constant in front of the leading order term in the bound of
$\VV_{\nu^{\ell,\ell-1}}[Y_\ell^n]$, namely the term
$\EE_{{\boldsymbol{\tilde \zeta}}}\big[\left(Q_\ell(\theta_\ell^{n+1})
  - Q_{\ell}(\theta_\ell')\right)^{2q_1}\big]^{1/q_1}$ in
\eqref{eq:b3}, depends on the difference between
$Q_\ell(\theta_\ell^{n+1})$ and $Q_{\ell}(\theta_\ell')$. In the case
of the pCN algorithm for the proposal distributions $q^{\ell-1}$ and
$q _\mathrm{ML}^{\ell,F}$ (as used in Section \ref{sec:num} below)
this difference will be small, since $\theta_\ell^n$ and
$\theta_\ell'$ will in general be very close to each other. However,
the difference is bounded from below and so we should eventually 
see the slower convergence rate for the variance as predicted by our theory.

\section{Numerics}
\label{sec:num}

In this section we describe the implementation details of the MLMCMC
algorithm and examine the performance of the method in estimating the
posterior expectation of some quantity of interest for our model problem
\eqref{eq:mod2}. 
We consider \eqref{eq:mod2} on the domain $D = (0, 1)^2$ with $f
\equiv 1$. On the lateral boundaries of the domain we choose Dirichlet
boundary conditions; on the top and bottom we choose Neumann conditions:
\begin{equation}
p|_{x_1 = 0} = 0, \quad p|_{x_1 = 1} = 1, \quad \frac{\partial p}{\partial \textbf{n}}\Big|_{x_2 = 0} = 0 \quad \mbox{and} \quad \frac{\partial p}{\partial \textbf{n}}\Big|_{x_2 = 1} = 0.
\end{equation}
The quantity of interest is the flux across the boundary at $x_1 = 1$, given by
\begin{equation}
\label{def:functional}
Q := -\int_0^1 k \frac{\partial p}{\partial x}\Big|_{x_1 = 1}\;dx_2.
\end{equation}
The (prior) permeability field $k$ is modelled as a log-normal random
field, with covariance function \eqref{eq:cov} with $r = 1$, $\sigma^2
= 1$ and $\lambda = 0.5$. The log-normal distribution is approximated
using truncated KL-expansion \eqref{truncated_kl} with an increasing 
number  $R_\ell$ of terms as $\ell$ increases. For $r=1$, the KL eigenfunctions
in \eqref{truncated_kl} are known explicitly \cite{cgst11}.

The model problem is discretised using piecewise linear FEs on a
uniform triangular mesh. The coarsest mesh consists of $m_0+1$ grid
points in each direction, with refined meshes containing $m_\ell +1 =
2^\ell m_0 + 1$ points, so that the total number of grid points on level $\ell$
is $M_\ell = (m_\ell+1)^2$.  All our algorithms have been implemented
within \texttt{freeFEM++} \cite{freeFEM}. As the linear solver for the
resulting linear equation system for each sample we used
\texttt{UMFPACK} \cite{Davis:04}.

\subsection{Implementation Details}

Let us first define two important quantities for the convergence
analysis of Metropolis-Hastings MCMC.\vspace{2ex}

\noindent{\em Effective sample size and integrated autocorrelation time.} 
Let $\{\theta^n\}_{n\ge 0}$ be the Markov chain
produced by Algorithm~1 and $\widehat{Q}^{\text{MC}}_N$ the 
resulting MCMC estimator defined in
\eqref{eq:mcmcest}. The integrated autocorrelation time $\tau_Q$ of the
correlated samples $Q^{n}_{M,R} := \mathcal{G}(X(\theta^n))$ produced
by Algorithm~1 is defined to be the ratio of the
asymptotic variance $\sigma_Q^2$ of the MCMC estimator 
$\widehat{Q}^{\text{MC}}_N$, defined in \eqref{def:assym_var}, and the 
actual variance $\VV_{\nu^{M,R}}[Q_{M,R}]$ of $Q_{M,R}$. If
\[
s_Q^2 := \frac{1}{N} \sum_{j=0}^N \left(
    Q_{M,R}^n - \widehat{Q}^{\text{MC}}_N \right)^2 
\]
denotes the sample variance, then a good estimate for $\tau_Q$, used
e.g. in \texttt{R},  is given by $\tau_Q = s_Q^2 / \rho(0)$,
where $\rho(0)$ is the so-called spectral density at
frequency zero. Details of a method for approximating
the  spectral density are given in \cite{HeiWel81}
(included  in \texttt{R} under the package \texttt{`coda'}). The
effective sample size is defined as $N^{\rm{eff}} :=
N/\tau_Q$. It represents the number of i.i.d. samples from $\nu_{M,R}$
that would lead to a Monte Carlo estimator with the same variance as 
$\widehat{Q}^{\text{MC}}_N$.\vspace{2ex}

\noindent
{\em Recursive independence sampling.} The final ingredient for our hierarchical multilevel MCMC algorithm 
is an efficient practical algorithm to obtain independent samples $\Theta_{\ell-1}^n$
from the coarse posterior $\nu^{\ell-1}$
which we need in Algorithm 2  in Section~3 to estimate
$\mathbb{E}_{\nu^\ell}[Q_\ell] - \mathbb{E}_{\nu^{\ell-1}}[Q_{\ell-1}]$. The algorithm is
summarised in Algorithm 3. 
\begin{figure}[t]
\begin{framed}
\noindent
{\bf ALGORITHM 3. (Recursive independence sampling)} \vspace{1ex} \\ 
Choose initial states $\Theta_{\ell-1}^0 = \widetilde{\Theta}_{\ell-1}^0, \ldots,
\widetilde{\Theta}_{0}^0$ such that
$\widetilde{\Theta}_{k,C}^0=\widetilde{\Theta}_{k-1}^0$ and
subsampling rates $t_k$, for all
$k=1,\ldots,\ell-1$. 
Then, for $j \geq 0$:
\begin{itemize} 
\item On level $0$: 
\begin{itemize}
\item Given $\widetilde{\Theta}_{0}^j$, generate $\widetilde{\Theta}_{0}'$ from a pCN proposal distribution.
\vspace{3mm}

\item Compute\vspace{-1.5ex} 
\begin{equation*}
\alpha^{0}(\widetilde{\Theta}_{0}' | \widetilde{\Theta}_{0}^j) = \min
\left\{ 1, \frac{\mathcal L_{0}(F_\mathrm{obs}\, |\, \widetilde
    \Theta_{0}')}{\mathcal L_{0}(F_\mathrm{obs}\, |\, \widetilde
    \Theta_{0}^j)} \right\}.  \hspace{25mm}\vspace{-0.5ex}
\end{equation*}

\item Set \;$\widetilde \Theta_{0}^{j+1}  = \widetilde \Theta_{0}'$\;
  with probability \;$\alpha^{0}(\widetilde \Theta_{0}'  | \widetilde
  \Theta_{0}^j)$. Set \;$\widetilde \Theta_{0}^{j+1}  = \widetilde
  \Theta_{0}^j$\; otherwise.
\end{itemize}
\item On level $k=1,\ldots,\ell-1$:
\begin{itemize}
\item Given $\widetilde \Theta_{k}^j$, let $\widetilde \Theta_{k,C}' =
  \widetilde \Theta_{k-1}^{(j+1)t_{k-1}}$ and generate $\widetilde \Theta_{k,F}'$ from a pCN proposal distribution.
\hspace*{-2mm}\vspace{3mm}
\item Compute \vspace{-1.5ex}
\begin{equation*}
\alpha^{\ell}_\mathrm{ML}(\widetilde \Theta_{k}' \, | \, \widetilde \Theta_{k}^j) \;=\; 
\min \left\{ 1, \frac{\mathcal L_k(F_\mathrm{obs}\, |\, \widetilde
    \Theta_k') \, \mathcal L_{k-1}(F_\mathrm{obs}\, |\, \widetilde
    \Theta_{k,C}^j)}{\mathcal L_k(F_\mathrm{obs}\, |\, \widetilde
    \Theta_k^j) \, \mathcal L_{k-1}(F_\mathrm{obs}\, |\, \widetilde \Theta_{k,C}')} \right\}. \hspace{25mm}
\end{equation*}\vspace{-1.5ex}

\item  Set \;$\widetilde \Theta_{k}^{j+1}  = \widetilde \Theta_{k}'$\;
  with probability \;$\alpha^{k}_{\text{ML}}(\widetilde \Theta_{k}'  | \widetilde
  \Theta_{k}^j)$. Set \;$\widetilde \Theta_{k}^{j+1}  = \widetilde
  \Theta_{k}^j$\; otherwise.
\vspace{3mm}

\end{itemize} 
\item Set \;$\Theta_{\ell-1}^{j+1} = \widetilde{\Theta}_{\ell-1}^{(j+1)t_{\ell-1}}$. 
\end{itemize}
\end{framed}
\end{figure}

We start on level 0 by creating a sufficiently long Markov chain 
$\{\widetilde{\Theta}_{0}^j\}_{j\ge0}$ using 
Algorithm 1 with pCN proposal distribution $q^0$~\cite{cds12} (see
\eqref{pCN_level0} below for details). 
Let $\widetilde Q_0^j := \mathcal{G}(p_0(\widetilde{\Theta}_{0}^j))$ be the
sample of the output quantity of interest associated with the $j$th
sample of the auxiliary chain  $\{\widetilde \Theta_0^j\}_{j\ge0}$ on
level~0. The
samples in this chain are correlated, but by subsampling it with a
sufficiently large rate $t_0 \in \mathbb{N}$, we obtain independent 
samples. The typical rule in statistics to achieve 
independence is to choose $t_0$ to be twice the integrated
autocorrelation time $\widetilde{\tau}_0$ of the Markov chain $\{\widetilde Q_0^j\}_{j\ge
  0}$. In practice, we found that much shorter subsampling rates 
were sufficient (see below). 

Then, on level $0 < k \le \ell-1$, we use the
independent samples 
created on level $k-1$ in Algorithm~2, to recursively create  a Markov chain 
$\{\widetilde{\Theta}_{k}^j\}_{j\ge0}$ on level $k$. The
proposal distribution $q^{k,F}$ for the modes that are added on level
$k$ is again chosen to be a pCN random walk  (see
\eqref{pCN_finemodes} below for details). We subsample this chain
again with sufficiently large rate $t_k \in \mathbb{N}$ to obtain
independent samples on level~$k$. Finally, we set $\Theta_{\ell-1}^n :=
\widetilde{\Theta}_{\ell-1}^{nt_{\ell-1}}$. In summary, to produce one
independent sample $\Theta_{\ell-1}^n$ on level $\ell-1$, we need to 
compute $T_k := \prod_{k'=k}^{\ell-1} t_{k'}$ samples, on each of levels 
$k=0,\ldots,\ell-1$. Since the acceptance probability
$\alpha^k_{\text{ML}}(\widetilde{\Theta}_{k}'|\widetilde{\Theta}_{k}^j)$
converges to $1$, as $k$ increases (cf.~Lemma \ref{thm:alphaconv}), and since we 
are using independent proposals from level $k-1$, the integrated 
autocorrelation times $\widetilde{\tau}_k$ of the auxiliary chains 
$\{\widetilde{\Theta}_{k}^j\}_{j\ge0}$, $k=1,\ldots,\ell-1$, converge
to $1$, i.e. the samples are essentially independent for large~$k$. As
a consequence $T_k$ is actually of the same order as the autocorrelation time 
of samples that Algorithm~1 with pCN proposals would 
produce on level $k$ (see below for more details). 

At the $j$th state of the auxiliary chain on level 0, the pCN proposal
from the standard multivariate normal prior distribution is generated as follows:
\begin{equation}
\label{pCN_level0}
(\widetilde \Theta'_{0})_i = \sqrt{1 - \beta_0^2} \, (\widetilde
\Theta^j_{0})_i + \beta_0\, \Psi_i\,, \quad i=1,\ldots,R_0\,.
\end{equation}
Here, $\Psi_i \sim \mathcal N(0,1)$ and $\beta_0$ is a tuning
parameter used to control the size of the step in the pCN random walk
\cite{cds12}. Similarly, the proposal $\widetilde \Theta'_{k,F}$ for
the fine modes at the $j$th state of the auxiliary chain on level
$k \in \{1,\ldots,L\}$ is generated by
\begin{equation}
\label{pCN_finemodes} 
(\widetilde \Theta'_{k,F})_i = \sqrt{1 - \beta_k^2} \, (\widetilde
\Theta^j_{k,F})_i + \beta_k \, \Psi_i\,, \quad i=1,\ldots,R_k - R_{k-1}\,.
\end{equation}
The actual values of $\beta_k = 0.1$, for all $k=0,\ldots,L$, that are
used in all the calculations that follow were chosen after carrying out a series of
preliminary tests to achieve ``good'' mixing properties. 

As in \eqref{eq:mcmcest}, in practice, the first $j_k^{0}$ samples from each
of the auxiliary chains are discarded by prescribing a ``burn-in''
period. We choose the length $j_k^{0}$ of the ``burn-in'' period on
level $k$ to be twice the integrated autocorrelation time
$\widetilde{\tau}_k$.\vspace{2ex}

\noindent
{\em Multilevel estimator.} We can now use the independent samples
$\Theta_{\ell-1}^n \sim \nu^{\ell-1}$ produced by Algorithm~3 above in
Algorithm~2 to produce samples $\theta_{\ell}^n$ of the fine chain on level
$\ell$, and thus samples $Y_\ell^n := \mathcal{G}(p_\ell(\theta_\ell^n)) - 
\mathcal{G}(p_{\ell-1}(\Theta_{\ell-1}^n))$ for the estimator $\widehat Y_{\ell,
  N_\ell}^{\mathrm{MC}}$ of $\EE_{\nu^\ell} [Q_\ell] -
\EE_{\nu^{\ell-1}}[ Q_{\ell-1}]$ in \eqref{eq:levelest}. The samples
for the estimator $\widehat Q_{0,N_0}^{\mathrm{MC}}$ on level $0$ are
produced with Algorithm~1 using again pCN-proposals. This completes the
definition of the multilevel MCMC estimator $\widehat
Q^\mathrm{ML}_{L,\{N_\ell\}}$ in \eqref{eq:mlmcmcest}.
It only remains to decide on an optimal sample size $N_\ell$ on each level that
will ensure that the total sampling error is below the prescribed
tolerance and that the total cost of the estimator is minimised.

Let $\tau_\ell$ be the integrated autocorrelation time of the chain
$Y_\ell^n$ (resp.~$Q_0^n$), for $\ell=1,\ldots,L$
(resp. $\ell=0$), and let $s_\ell^2$ be the sample variance on level
$\ell$. Then $N^{\rm{eff}}_{\ell} :=  N_{\ell} /
\tau_\ell$ is the effective sample size on level $\ell$ and
$s_\ell^2/N^{\rm{eff}}_{\ell}$ is an estimate of the variance
of the estimator $\widehat Y_{\ell,
  N_\ell}^{\mathrm{MC}}$. Our aim
is to achieve the following bound on the 
total sampling error for the multilevel MCMC estimator:
\begin{equation}
\sum^{L}_{\ell=0}\frac{s_\ell^2}{N_{\ell}^{\rm{eff}}} \leq \frac{\varepsilon^2}{2},
\label{eqn:samplingConstraint}
\end{equation}
for some prescribed tolerance $\varepsilon$. In what follows, we
will choose $\varepsilon$ such that the bias error on level $L$ is
$\frac{\varepsilon^2}{2}$ and thus the two contributions to the mean
square error in \eqref{eq:mse_multi} are balanced.

To decide on a cost-optimal strategy for the choice of the $N_\ell$,
we first need to discuss the cost per sample.  Recall that 
$\mathcal{C}_\ell$ denotes the cost to evaluate $Q_\ell$ for a single 
sample $\Theta_\ell$ from the prior on level~$\ell$. 
However, to quantify the cost of the estimator $\widehat Y_{\ell,
  N_\ell}^{\mathrm{MC}}$ on level $\ell$, we also need to take all the 
samples in the auxiliary chains 
on the coarser levels in Algorithm~3 into
account, as well as the integrated
autocorrelation time $\tau_\ell$ of the chain $\{Y_\ell^n\}$. Recalling that $t_k$
is the subsampling rate on level $k$ in Algorithm~3 and that $T_k = \prod_{k'=k}^{\ell-1}
t_{k'}$, the total cost to produce one independent (effective) sample is
\begin{equation}
\label{eq:effectivecost}
\mathcal{C}^{\text{eff}}_\ell := \lceil\tau_\ell\rceil
\left(\mathcal{C}_\ell + \sum_{k=1}^{\ell-1} T_k \, \mathcal{C}_k
\right)\,.
\end{equation}

As in the case of standard multilevel MC with i.i.d. samples, the
total cost of the multilevel estimator is minimised, subject to the
constraint~\eqref{eqn:samplingConstraint}, when the effective 
number of samples on each level satisfies
\begin{equation}
N^{\text{eff}}_\ell = \frac{2}{\varepsilon^2}\left(\sum_{\ell=0}^{L}\sqrt{s_\ell^2\mathcal C^{\text{eff}}_\ell}\right)\sqrt{\frac{s_\ell^2}{\mathcal C^{\text{eff}}_\ell}}
\label{eqn:optimumNl}
\end{equation}
as described in \cite{giles08,cgst11}. In practice, the optimal number
of samples can be estimated adaptively after an initial number of
samples to get an estimate for $s_\ell^2$  
(see again \cite{giles08,cgst11} for standard MLMC).

In all calculations which follow we simultaneously run $P$ parallel chains. This allows for an efficient parallelisation and aids exploration of multi-modal posterior distributions. Furthermore the calculation of the
total sampling error \eqref{eqn:samplingConstraint} is simplified. The
parallel chains provide $P$ independent estimates for $\widehat
Y_{\ell}^{\mathrm{MC}}$. Therefore, using standard statistical tools,
the sampling error on each level can be calculated without the
need for accurate estimates of the integrated autocorrelation times. 
For the implementation considered
here we chose $P = 128$ and distributed the computations across 128 processors.

\subsection{Two-Level Results}\label{sec:twoleveltests}

We start with a two level test to investigate the additional bias created in Algorithm 2 due to the dependence of 
the coarse samples from the recursive subsampling procedure in
Algorithm 3 and how that bias depends on the subsampling rate $t_k$. We choose two grids with $m_0 =
8$ and $m_1 = 16$ and fix the numbers of KL modes to be $R_0 = R_1 =
20$. The data is generated synthetically from a single random sample
from the prior distribution computed on grid level 4, i.e. with $m_4 =
128$. The observations $F_{obs}$ are taken to be the
pressure values at 16 uniformly spaced points interior to the
domain. The data fidelity is set to $\sigma_F^2 = 10^{-4}$ on both
levels. 

\begin{figure}
\includegraphics[width=0.48\linewidth]{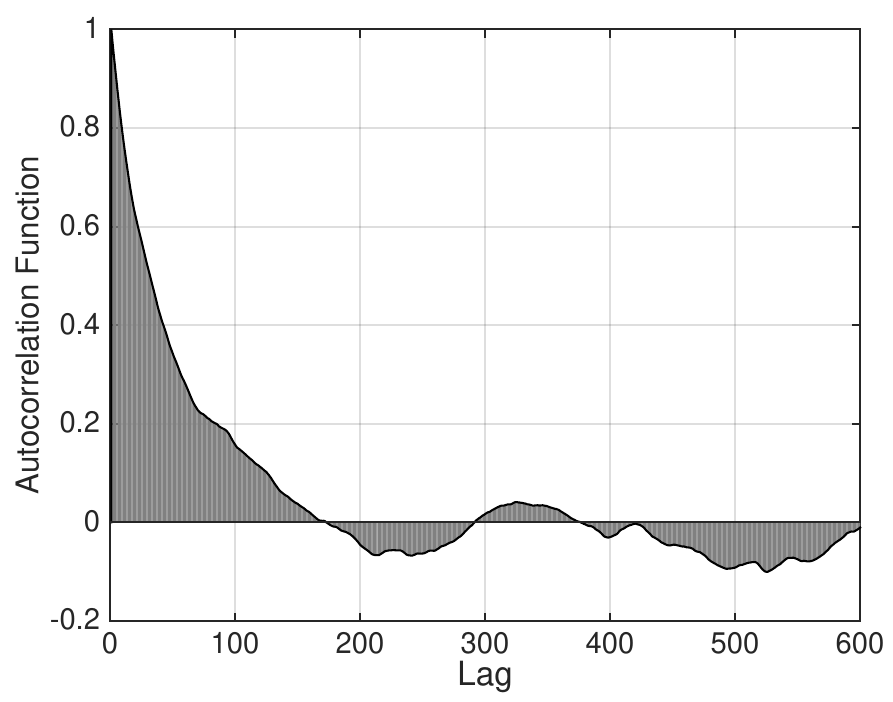} \hfill\includegraphics[width=0.44\linewidth]{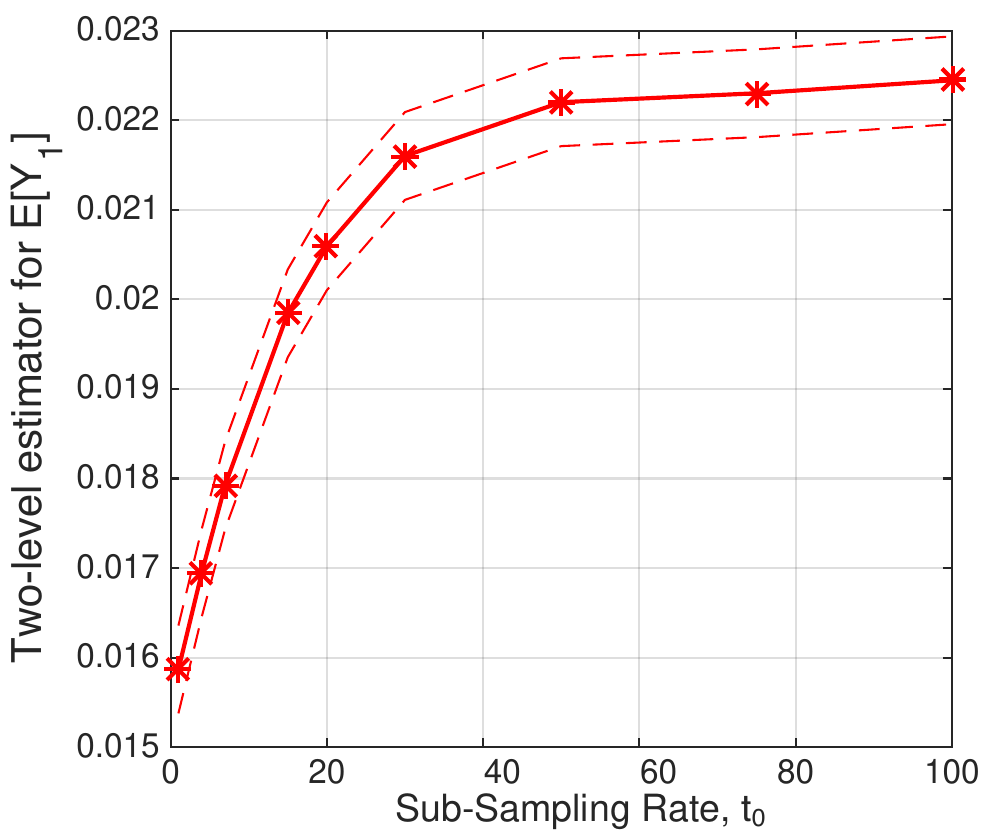}
\caption{Left: Autocorrelation function for a typical
  (burnt-in) coarse level chain $\{Q_0^{n}\}$ with an integrated
  autocorrelation time of $\tau_0 \approx 86$. Right: Plot of $\mathbb
  E\big[\widehat Y^{\text{MC}}_{1}\big]$ against subsampling rate $t_0$; the solid line shows
  the computed results whilst the dashed lines give the two-sided 95$\%$ confidence interval.}
\label{fig:biasError}
\end{figure}
We first computed the autocorrelation function for a typical (burnt-in)
coarse level chain $\{Q_0^n\}$ (see Fig. \ref{fig:biasError}(left)) and note that
the integrated autocorrelation time is approximately $\tau_0 \approx
86$ in this case. We then ran Algorithms 2 and 3, for different
subsampling rates from $t_0 = 1$ to $100 >
\tau_0$, until the standard
error for the estimator $\widehat{Y}^{\text{MC}}_{1}$ reached a
prescribed tolerance of $\varepsilon = 2.5 \times 10^{-4}$. Fig.~\ref{fig:biasError}(right) shows the expected value of
$\mathbb E_{\boldsymbol{\Theta}_1}[\widehat{Y}^{\text{MC}}_{1}]$ as a function of $t_0$, as
well as the two-sided $95\%$ confidence interval, i.e. $\mathbb
E_{\boldsymbol{\Theta}_1}[\widehat{Y}^{\text{MC}}_1] \pm
1.96\,\varepsilon$. We note that  $\mathbb
E_{\nu^1}[Q_1]  - \mathbb
E_{\nu^0}[Q_0] \approx \mathbb
E_{\{\Theta^n_1\}}[\widehat Q^{\text{MC}}_1] - \mathbb
E_{\{\Theta^n_0\}}[\widehat Q^{\text{MC}}_0] \approx 0.0222$, calculated
from two independent standard MCMC runs to a tolerance of $\varepsilon =
2.5 \times 10^{-5}$ on each level. 

We note that, for the example considered here, the additional bias
error due to the dependence of the samples is less than $30\%$ 
even if no subsampling is used (i.e. $t_0 = 1$). In practice, a value of $t_0=50$ would 
be sufficient to reduce the bias to a negligible amount ($< 1\%$), given all
the other bias errors due to FE discretisation, KL truncation and
Metropolis-Hastings sampling. However, to be on the safe side for all
the calculations that follow we take the subsampling rate equal to the
smallest integer that is bigger than our estimate of the integrated 
autocorrelation time, i.e. $t_\ell = \lceil\widetilde{\tau}_\ell\rceil$.

\subsection{Comparison of MLMCMC with a standard
  single-level MCMC estimator}\label{sec:MLMCMCtests}

We now test the full MLMCMC Algorithm, using the same coarsest grid with
$m_0 = 8$ and considering up to five levels in our method with a uniformly
increasing number of KL modes across the levels from $R_0 = 50$ to
$R_4 = 150$. As for the two level example, the data is generated
synthetically from a single random sample from the prior distribution
on level 4, see Fig.~\ref{fig:data_posteriorsample}(left). We note that since $R_4 =150$
here, the data differs slightly from that used in the two-level 
results in Sect.~\ref{sec:twoleveltests} (although we used the
same random numbers for the first 20 KL modes). The fidelity parameter 
was again chosen to be $\sigma_{F,\ell}^2 = 10^{-4}$, for all
$\ell=0,\ldots,4$. A typical sample from the posterior distribution on
grid level 4, produced by our multilevel algorithm, is shown in 
Fig.~\ref{fig:data_posteriorsample}(right).
\begin{figure}
 \includegraphics[width = 0.49\linewidth]{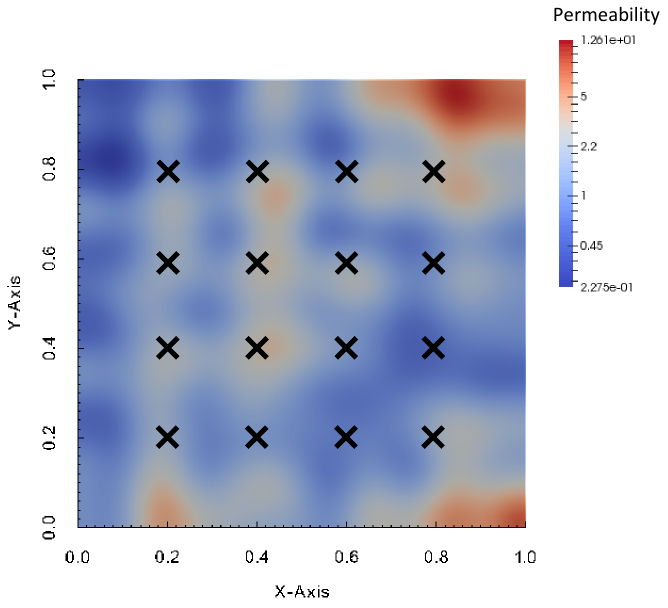} \hfill
\includegraphics[width = 0.49\linewidth]{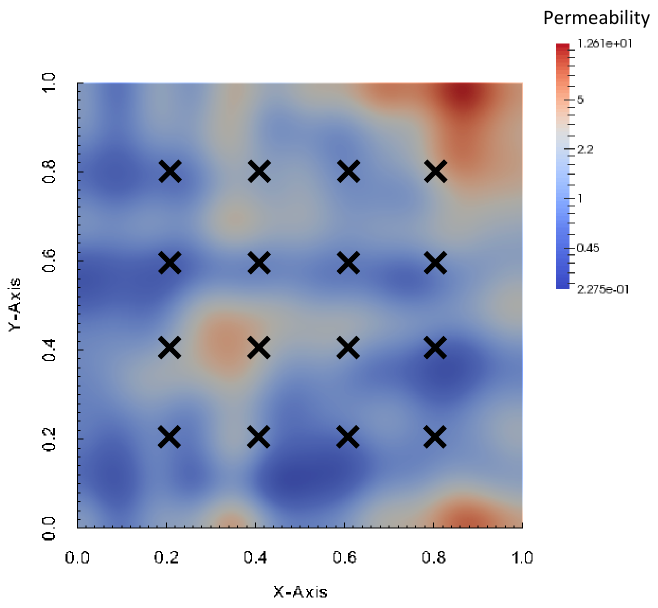}
 \caption{\label{fig:data_posteriorsample} Left: Synthetic data used
   in Section \ref{sec:MLMCMCtests}. Right: Posterior sample created by our
   algorithm on grid level 4. For both plots, data points are marked by crosses.}
 \end{figure}

\begin{figure}
 \includegraphics[width =
 0.45\linewidth]{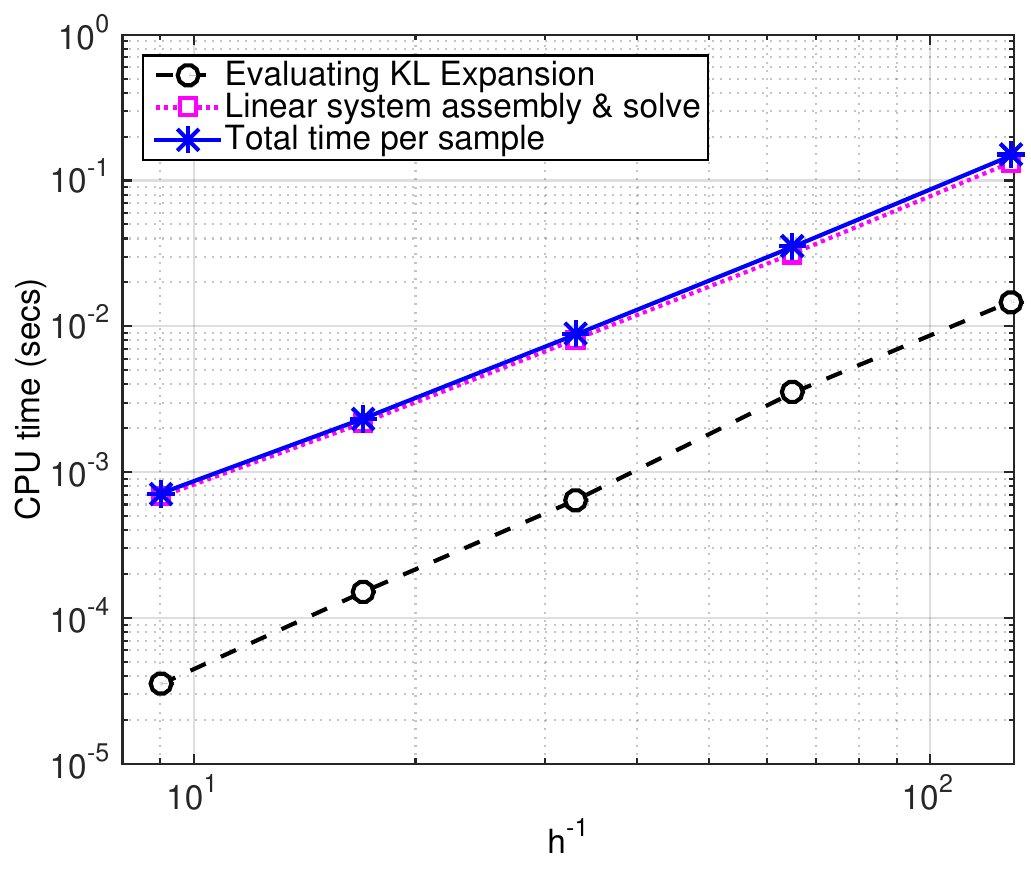} \hfill \includegraphics[width = 0.45\linewidth]{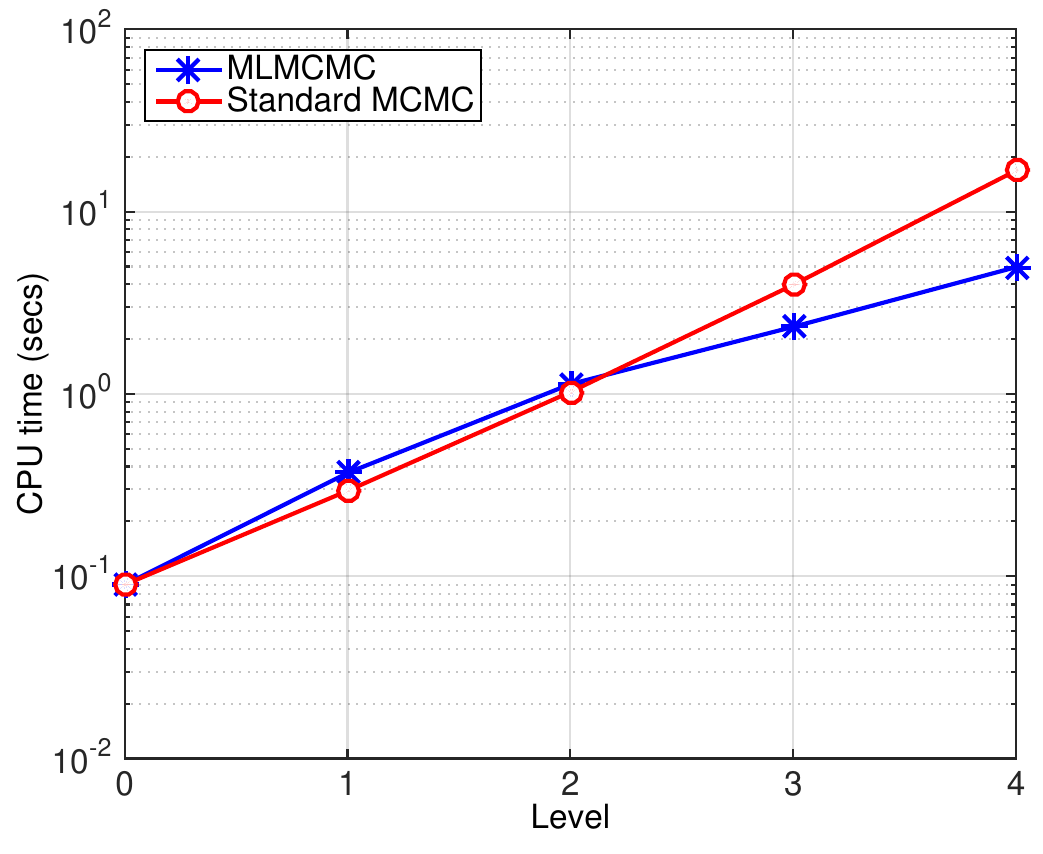}
 \caption{\label{fig:sampleCosts} Left: Cost (CPU time in seconds)
   to compute one sample of $Q_h$ as a function of $h$. Right: Cost
   $\mathcal C^{\text{eff}}_\ell$ per independent sample on level $\ell$.}
 \end{figure}
 We compare the performance of our new multilevel method to standard
Metropolis-Hastings MCMC with pCN proposal distribution (again with
tuning parameter $\beta_\ell = 0.1$). The cost $\mathcal C_\ell$ to
compute one individual sample of $Q_\ell$ on level $\ell$ with our
code is shown in actual CPU
time in Fig.~\ref{fig:sampleCosts}(left), obtained on a 2.4GHz Intel
Core i7 processor. The cost in
{\tt FreeFEM++} is dominated by the assembly of the FE stiffness
matrix and so it grows like $\mathcal{O}(h_\ell^{-2}) =
\mathcal{O}(M_\ell)$. We believe that this behaviour is
representative for problems of this size when the uniform grid
structure is not exploited in the assembly process and that these CPU times are
competitive. For larger problem sizes, the cost of the
linear solver will become the dominant part. However, for the MLMCMC
algorithm we are really interested in the cost $\mathcal
C^{\text{eff}}_\ell$ defined in \eqref{eq:effectivecost} to compute
one independent sample on level $\ell$ using Algorithms 2 and 3 with 
$t_k = \lceil \widetilde{\tau}_k \rceil$. These times are shown in
Fig.~\ref{fig:sampleCosts}(right). They are compared to the cost to
produce  one independent
sample on level $\ell$ using the standard MCMC Algorithm~1. The
integrated autocorrelation times $\widetilde{\tau}_\ell$ for the auxiliary
chains $\{\widetilde{Q}_\ell^n\}$ on each level in our example are
given in Tab.~\ref{tab:autocorrelationlengths}. Note that since the
coarse samples are (essentially) independent, the integrated
autocorrelation times $\tau_\ell$ for the chains $\{Y_\ell^n\}$ are
almost identical, i.e. $\tau_\ell \approx \widetilde{\tau}_\ell$.
\begin{table}[h]
\begin{center}
\begin{tabular}{|c|c|c|c|c|c|}
\hline
Level                                                                & 0 & 1 & 2 & 3 & 4 \\ \hline
$\widetilde{\tau}_\ell$ & 136.23  & 3.66  & 2.93  & 1.46  & 1.23  \\ \hline
\end{tabular}
\caption{\label{tab:autocorrelationlengths} Integrated autocorrelation
  times of the auxiliary chains $\{\widetilde{Q}^n_\ell\}$ on levels $\ell=0,\ldots,4$.} 
\end{center}
\end{table}
\renewcommand{\floatpagefraction}{.8}

In Fig.~\ref{fig:costcomparison} we now compare the performance of 
our MLMCMC method
with finest level $L$ varying from $1$ to $4$ with standard MCMC on
the same level. The tolerance $\varepsilon_L$ for each of the cases is 
chosen such that the the bias error is less than
$\varepsilon_L/\sqrt{2}$, leading to $\varepsilon_1 = 0.04$,
$\varepsilon_2 = 0.017$, $\varepsilon_3 = 0.013$ and
$\varepsilon_4 = 0.0067$, respectively. The estimated bias error
decays with about $\mathcal{O}(h)$ which is faster than what we would
expect for the functional in \eqref{def:functional} which does not
satisfy Assumption A2 (see
\cite{tsgu13}). It is likely that this is because the second term in
\eqref{eq:bias_triangle}, i.e. the bias error in the posterior distribution,
dominates. That bias error is due to the FE approximation of 
pressure evaluations at points here, which are
expected to converge with $\mathcal{O}(h \log|h|))$ (see
\cite{teckentrup_thesis}). The slight variation in the convergence
rate could mean that some features in the posterior were only picked
up on a sufficiently fine grid. The optimal numbers $N^{\text{eff}}_\ell$ of (independent)
samples on each level are chosen according to formula
\eqref{eqn:optimumNl}. They are plotted in
Fig.~\ref{fig:costcomparison}(left). Please note that these are
numbers of independent samples. The total number of samples computed
on the coarser levels is much larger. For example, for the four
level estimator we needed about $4 \times 10^7$ actual PDE solves for
all the auxiliary chains on level 0 combined. However, each of these
solves is about 250 times cheaper than a solve on level 4. Because
$\tau_4 \approx \widetilde{\tau}_4 = 1.23$, we see from
Fig.~\ref{fig:costcomparison}(left) that we need only about 562 PDE
solves on level 4. These are huge savings against standard MCMC which
requires about $4 \times 10^6$ solves on level 4 to achieve the same sampling
error. We can see this clearly in the overall cost comparison in 
Fig.~\ref{fig:costcomparison}(right). The gains are even more pronounced
if we relax the overly conservative choice of $t_k = \lceil
\widetilde{\tau}_k \rceil$ for the subsampling rates. 
\begin{figure}
\includegraphics[width = 0.52\linewidth]{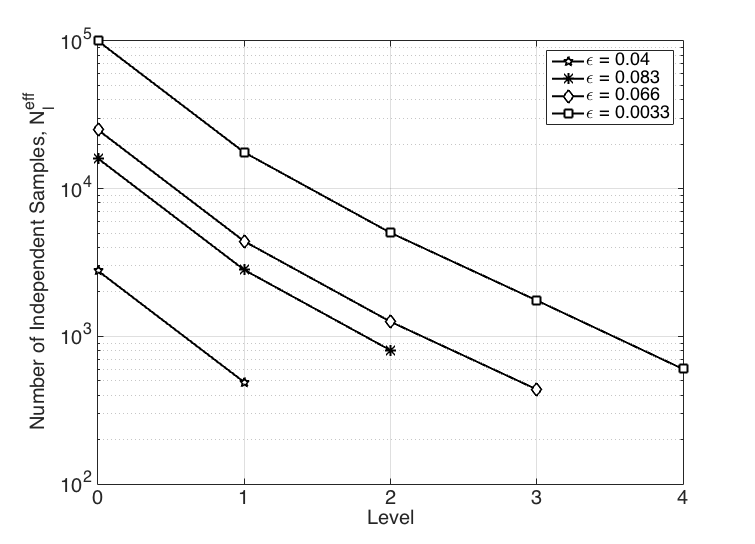} \hfill
\includegraphics[width = 0.4\linewidth]{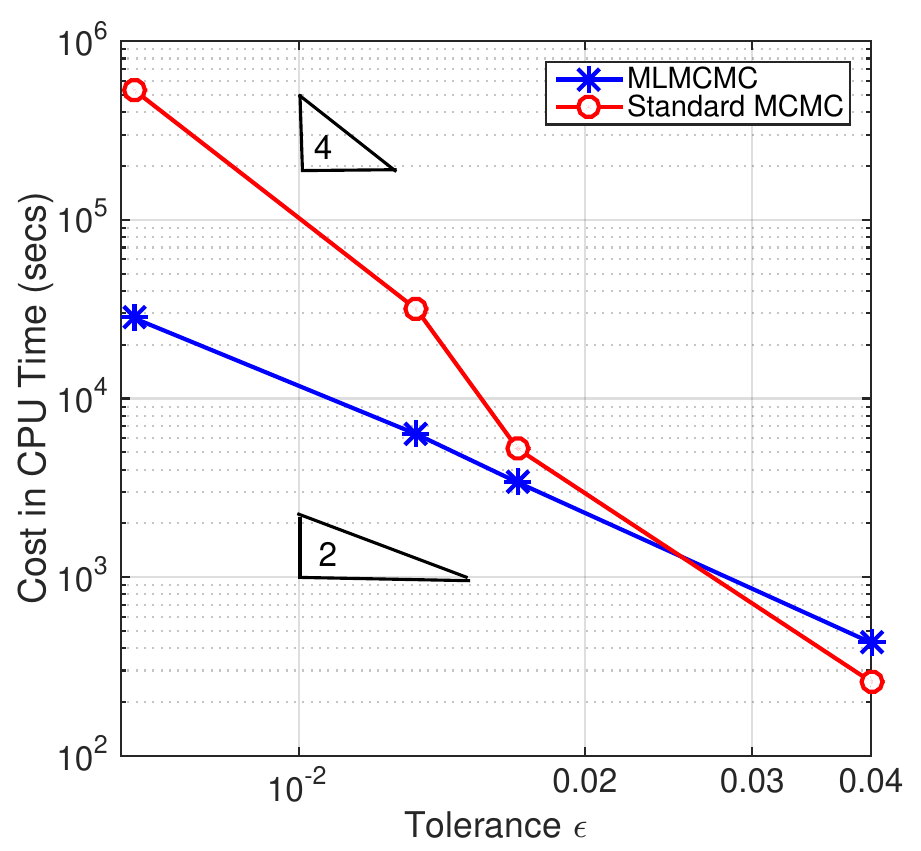}
\caption{\label{fig:costcomparison}Left: Number of independent samples
  $N_\ell^{\text{eff}}$ on each level for four different
  tolerances. Right: Total cost (in
  seconds) for the multilevel and the
  single-level estimators plotted against tolerance $\epsilon$.}
\end{figure}

In our final Fig.~\ref{fig:variance_alpha}, we confirm our theoretical results and plot
our estimates for $\VV_{\nu^{\ell,\ell-1}}\left[Y_\ell^{n} \right]$
(left) and for $\EE_{\boldsymbol{\zeta}}\big[(1 -
\alpha^\ell_\mathrm{ML}(\theta_\ell'|\theta_\ell^{n}))\big]$
(right). Ignoring the last data point in each of the plots, which
seem to be outliers, the variance seems to converge with almost
$\mathcal{O}(h^2)$ and the multilevel rejection probability
slightly faster than $\mathcal{O}(h)$. We are not sure whether this
means that the bounds in Lemma \ref{lem:varconv} and in Lemma
\ref{thm:alphaconv} are both
slightly pessimistic or whether this is just some pre-asymptotic 
behaviour.
\begin{figure}
\includegraphics[width = 0.45\linewidth]{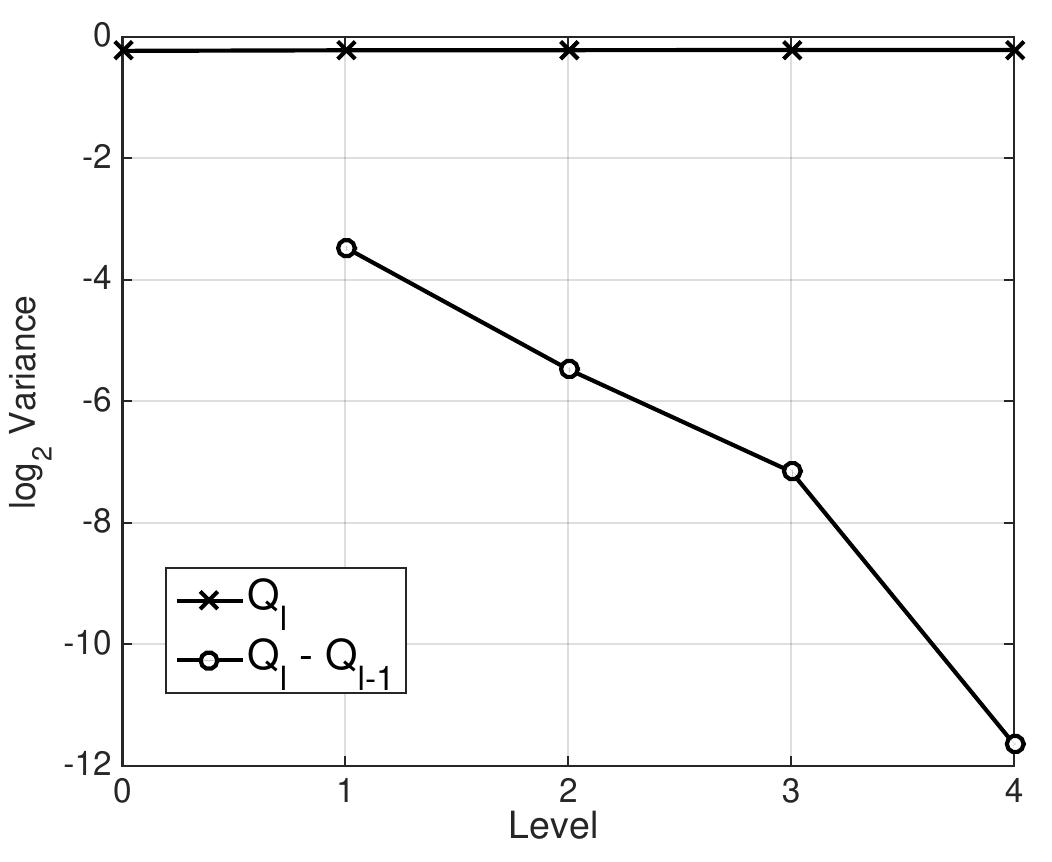} \hfill
\includegraphics[width = 0.45\linewidth]{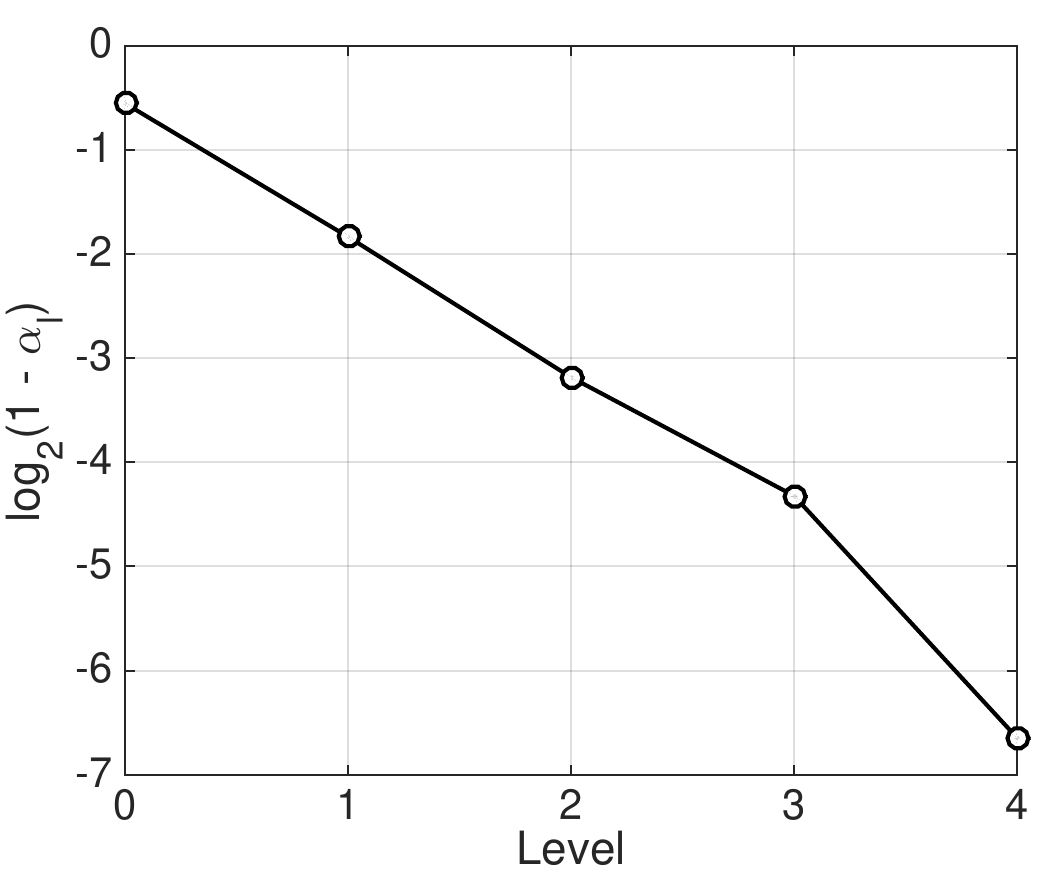}
\caption{\label{fig:variance_alpha}Convergence plots for
  $\VV_{\nu^{\ell,\ell-1}}\left[Y_\ell^{n} \right]$ and $\EE_{\boldsymbol{\zeta}}\big[(1 -
\alpha^\ell_\mathrm{ML}(\theta_\ell'|\theta_\ell^{n}))\big]$.}
\end{figure}

\begin{remark}\rm
It is worth to point out that the recursive independence sampling
in Algorithm~3 also brings significant savings if used to produce
proposals for a standard MCMC algorithm, as the comparison of the cost
per independent sample in Fig.~\ref{fig:sampleCosts}(right) clearly
shows. This is related to the delayed acceptance method of
\cite{cf05}. The multilevel approach also provides a very efficient
burn-in method, due to the significantly reduced integrated
autocorrelation times on the finer levels and since most of the
burn-in happens on the coarsest level. This is related to the approach
in \cite{ehl06}. 
\end{remark}

\section{Conclusion}
Bayesian inverse problems in large scale applications are often too costly to solve using conventional Metropolis-Hastings MCMC algorithms due to the high dimension of the parameter space and the large cost of computing the likelihood. In this paper, we employed a hierarchy of computational models to define a novel multilevel version of a Metropolis-Hastings algorithm, leading to significant reductions in computational cost. The main idea underlying the cost reduction is to build estimators for the difference in the quantity of interest between two successive models in the hierarchy, rather than estimators for the quantity itself. The new algorithm was then analysed and implemented for a single-phase Darcy flow problem in groundwater modelling, confirming the effectiveness of the algorithm. 

The algorithm presented in this paper is not reliant on the specific computational model underlying the simulations, and is generally applicable. The underlying computational model will in general influence the convergence rates $\alpha, \alpha', \beta$ and $\beta'$ of the discretisation errors, and the growth rate $\gamma$ of the cost of the likelihood computation (cf Theorem \ref{thm:main}), which in turn govern the cost of the standard and multilevel Metropolis-Hastings algorithms. The gain to be expected from employing the multilevel algorithm is always significant, and the gain is in fact larger for more challenging model problems, where the values of $\alpha, \alpha', \beta$ and $\beta'$ are small and $\gamma$ is large.

The algorithm also allows for the use of a variety of proposal
distributions. The crucial result in this context is the convergence
of the multilevel acceptance probability to $1$ (cf.~Lemma
\ref{thm:alphaconv}), which in general has to be verified for each
proposal distribution individually, but is expected to hold for most
proposal distributions.\vspace{2ex}

\noindent
{\bf Acknowledgement.} Big thanks go to Panayot Vassilevski who
  initiated and financially supported this work during two visits of
  Scheichl and Teckentrup at Lawrence Livermore National
  Labs (LLNL), California. He was involved in 
  most of the original discussions
  about this method. Christian Ketelsen was postdoctoral researcher
  under his supervision at LLNL under Contract DE-AC52-07A27344 at the
  time. We would also like to
  particularly thank Finn Lindgren and Rob Jack for spotting an error
  in our original version of Lemma \ref{lem:all} and for helping us
  to find a fix.

\small
\bibliographystyle{plain}
\bibliography{bibMLMC}
\end{document}